\documentclass[final,leqno]{siamltex}

\usepackage{graphicx}
\usepackage{amsfonts}
\usepackage{amsmath}
\usepackage[mathcal]{euscript}      
\usepackage{bm}                     
\usepackage{algorithm, algorithmic} 
\usepackage{color}
\usepackage{epstopdf}
\usepackage{amssymb}
\usepackage{mathrsfs}
\usepackage{listings}               
\usepackage{xcolor}

\title{Computing Eigenvalues of Large Scale Sparse Tensors Arising from
  a Hypergraph\thanks{Version 1.0. Data: \today.}}

\author{
  {Jingya Chang}\thanks{%
    Department of Applied Mathematics, The Hong Kong Polytechnic University,
    Hung Hom, Kowloon, Hong Kong; and
    School of Mathematics and Statistics, Zhengzhou University, Zhengzhou 450001, China
    ({\tt jychang@zzu.edu.cn}).} \and
  {Yannan Chen}\thanks{%
    School of Mathematics and Statistics, Zhengzhou University, Zhengzhou 450001, China
    ({\tt ynchen@zzu.edu.cn}).
    This author was supported by the
    National Natural Science Foundation of China (Grant No. 11401539, 11571178),
    the Development Foundation for Excellent Youth Scholars of
    Zhengzhou University (Grant No. 1421315070), and
    the Hong Kong Polytechnic University Postdoctoral Fellowship.} \and
  {Liqun Qi}\thanks{%
    Department of Applied Mathematics, The Hong Kong Polytechnic University,
    Hung Hom, Kowloon, Hong Kong ({\tt maqilq@polyu.edu.hk}).
    This author's work was partially supported by the Hong Kong Research Grant Council
    (Grant No. PolyU 501212, 501913, 15302114 and 15300715).}
}

\begin{document}

\maketitle

\begin{abstract}
  The spectral theory of higher-order symmetric tensors is an important tool
  to reveal some important properties of a hypergraph via its
  adjacency tensor, Laplacian tensor, and signless Laplacian tensor.
  Owing to the sparsity of these tensors, we propose an efficient approach to
  calculate products of these tensors and any vectors.
  Using the state-of-the-art L-BFGS approach,
  we develop a first-order optimization algorithm
  for computing H- and Z-eigenvalues of these large scale sparse tensors (CEST).
  With the aid of the Kurdyka-{\L}ojasiewicz property, we prove that
  the sequence of iterates generated by CEST converges to an eigenvector of the tensor.
  When CEST is started from multiple randomly initial points,
  the resulting best eigenvalue could touch the extreme eigenvalue with a high probability.
  Finally, numerical experiments on small hypergraphs show that
  CEST is efficient and promising. Moreover, CEST is capable of computing
  eigenvalues of tensors corresponding to a hypergraph with millions of vertices.
\end{abstract}



\pagestyle{myheadings}
\thispagestyle{plain}
\markboth{JINGYA CHANG, YANNAN CHEN, AND LIQUN QI}{COMPUTING EIGENVALUES OF SPARSE TENSORS}


\newtheorem{Theorem}{Theorem}[section]
\newtheorem{Definition}[Theorem]{Definition}
\newtheorem{Lemma}[Theorem]{Lemma}

\newcommand{\REAL}{\mathbb{R}}
\newcommand{\T}{\top}
\newcommand{\st}{\mathrm{s.t.}}
\newcommand{\diff}{\mathrm{d}}
\newcommand{\vt}[1]{{\bf #1}}
\newcommand{\x}{\vt{x}}
\newcommand{\g}{\vt{g}}
\newcommand{\s}{\vt{s}}
\newcommand{\y}{\vt{y}}
\newcommand{\z}{\vt{z}}
\newcommand{\p}{\vt{p}}
\newcommand{\dvec}{\vt{d}}
\newcommand{\Adj}{\mathcal{A}}
\newcommand{\Dag}{\mathcal{D}}
\newcommand{\Lap}{\mathcal{L}}
\newcommand{\sLp}{\mathcal{Q}}
\newcommand{\Ten}{\mathcal{T}}
\newcommand{\spT}{\mathcal{S}}
\newcommand{\Iid}{\mathcal{I}}
\newcommand{\Eid}{\mathcal{E}}
\newcommand{\Btn}{\mathcal{B}}

\newcommand{\SPHERE}{\mathbb{S}^{n-1}}
\newcommand{\Path}[1]{\texttt{#1}}

\section{Introduction}

Since 1736, Leonhard Eular posed a problem called ``seven bridges of
K\"{o}nigsberg'', graphs and hypergraphs have been used to model
relations and connections of objects in science and engineering,
such as molecular chemistry \cite{KoS-95,KoS-98}, image processing
\cite{GWTJD-12,YTW-12}, networks \cite{KHT-09,GZCN-09}, scientific
computing \cite{FMS-10,KPCA-12}, and very large scale integration
(VLSI) design \cite{KAKS-99}. For large scale hypergraphs, spectral
hypergraph theory provides a fundamental tool. For
instance, hypergraph-based spectral clustering has been used in
complex networks \cite{MN-12}, date mining \cite{LHSDZ-14}, and
statistics \cite{RCY-11,LR-15}. In computer-aided design
\cite{ZSC-99} and machine learning \cite{GD-15}, researchers
employed the spectral hypergraph partitioning. Other applications
include the multilinear pagerank \cite{GLY-15} and estimations of
the clique number of a graph \cite{BP-09,XQ-15}.







Recently, spectral hypergraph theory is proposed to explore
connections between the geometry of a uniform hypergraph and H- and
Z-eigenvalues of some related symmetric tensors. Cooper and Dutle
\cite{CoD-12} proposed in 2012 the concept of adjacency tensor for a
uniform hypergraph. Two years later, Qi \cite{Qi-14} gave
definitions of Laplacian and signless Laplacian tensors
associated with a hypergraph. When an even-uniform hypergraph is connected,
the largest H-eigenvalues of the Laplacian and signless
Laplacian tensors are equivalent if and only if the hypergraph is
odd-bipartite \cite{HQX-15}. This result gives a certification to check whether a
connected even-uniform hypergraph is odd-bipartite or not.

We consider the problem of how to compute H- and Z-eigenvalues of the adjacency tensor,
the Laplacian tensor, and the signless Laplacian tensor arising from a uniform hypergraph.
Since the adjacency tensor and the signless Laplacian tensor are symmetric and nonnegative,
an efficient numerical approach named the Ng-Qi-Zhou algorithm \cite{NQZ-09} could
be applied for their largest H-eigenvalues and associated eigenvectors.
Chang et al. \cite{CPZ-11} proved that the Ng-Qi-Zhou algorithm converges
if the nonnegative symmetric tensor is primitive.
Liu et al. \cite{LZI-10} and Chang et al. \cite{CPZ-11} enhanced the Ng-Qi-Zhou algorithm
and proved that the enhanced one converges if the nonnegative symmetric tensor is irreducible.
Friedland et al. \cite{FGH-13} studied weakly irreducible nonnegative symmetric tensors
and showed that the Ng-Qi-Zhou algorithm converges with an R-linear convergence rate
for the largest H-eigenvalue of a weakly irreducible nonnegative symmetric tensor.
Zhou et al. \cite{ZQW-13a,ZQW-13b} argued that the Ng-Qi-Zhou algorithm is Q-linear convergence.
They refined the Ng-Qi-Zhou algorithm and reported that they could obtain
the largest H-eigenvalue for any nonnegative symmetric tensors.
A Newton's method with locally quadratic rate of convergence is established by Ni and Qi \cite{NiQ-15}.

With respect to the eigenvalue problem of general symmetric tensors,
there are two sorts of methods.
The first one could obtain all (real) eigenvalues of a tensor with only several variables.
Qi et al. \cite{QWW-09} proposed a direct approach based on the resultant.
An SDP relaxation method coming from polynomial optimization was established by
Cui et al. \cite{CDN-14}. Chen et al. \cite{CHZ-15} preferred to use homotopy methods.
Additionally, mathematical softwares \emph{Mathematica} and \emph{Mapple}
provide respectively subroutines ``\texttt{NSolve}'' and ``\texttt{solve}'' which could solve
polynomial eigen-systems exactly.
However, if we apply these methods for eigenvalues of a symmetric tensor with
dozens of variables, the computational time is prohibitively long.

The second sort of methods turn to compute an (extreme) eigenvalue of a symmetric tensor,
since a general symmetric tensor has plenty of eigenvalues \cite{Qi-05} and
it is NP-hard to compute all of them \cite{HiL-13}.
Kolda and Mayo \cite{KoM-11,KoM-14} proposed a spherical optimization model
and established shifted power methods. Using fixed point theory,
they proved that shifted power methods converge to an eigenvalue
and its associated eigenvector of a symmetric tensor.
For the same spherical optimization model, Hao et al. \cite{HCD-15} prefer to
use a faster subspace projection method.
Han \cite{Han-13} constructed an unconstrained merit function that is indeed
a quadratic penalty function of the spherical optimization.
Preliminary numerical tests showed that these methods could compute
eigenvalues of symmetric tensors with dozens of variables.

How to compute the (extreme) eigenvalue of the Laplacian tensor coming from
an even-uniform hypergraph with millions of vertices?
It is expensive to store and process a huge Laplacian tensor directly.


In this paper, we propose to store a uniform hypergraph by a matrix,
whose row corresponds to an edge of that hypergraph.
Then, instead of generating the large scale Laplacian tensor of the hypergraph explicitly,
we give a fast computational framework for products of the Laplacian tensor and any vectors.
The computational cost is linear in the size of edges and quadratic in
the number of vertices of an edge. So it is cheap.
Other tensors arising from a uniform hypergraph, such as
the adjacency tensor and the signless Laplacian tensor,
could be processed in a similar way.
These computational methods compose our main motivation.

Since products of any vectors and large scale tensors
associated with a uniform hypergraph could be computed economically,
we develop an efficient first-order optimization algorithm
for computing H- and Z-eigenvalues of adjacency, Laplacian, and signless Laplacian
tensors corresponding to the even-uniform hypergraph.
In order to obtain an eigenvalue of an even-order symmetric tensor, we minimize
a smooth merit function in a spherical constraint,
whose first-order stationary point is an eigenvector associated with a certain eigenvalue.
To preserve the spherical constraint, we derive an explicit formula
for iterates using the Cayley transform.
Then, the algorithm for a spherical optimization
looks like an unconstrained optimization.
In order to deal with large scale problems,
we explore the state-of-the-art L-BFGS approach to generate
a gradient-related direction and the backtracking search
to facilitate the convergence of iterates.
Based on these techniques, we obtain the novel algorithm (CEST) for
computing eigenvalues of even-order symmetric tensors.
Due to the algebraic nature of tensor eigenvalue problems,
the smooth merit function enjoys the Kurdyka-{\L}ojasiewicz (KL) property.
Using this property, we confirm that the sequence of iterates generated by CEST
converges to an eigenvector corresponding to an eigenvalue.
Moreover, if we start CEST from multiple initial points
sampled uniformly from a unit sphere, it can be proved that the resulting best merit function value
could touch the extreme eigenvalue with a high probability.

Numerical experiments show that the novel algorithm CEST is
dozens times faster than the power method for
eigenvalues of symmetric tensors related with small hypergraphs.
Finally, we report that CEST could compute H- and Z-eigenvalues and
associated eigenvectors of symmetric tensors involved in
an even-uniform hypergraph with millions of vertices.

The outline of this paper is drawn as follows.
We introduce some latest developments on spectral hypergraph theory in Section 2.
Section 3 address the computational issues on products of a vector and
large scale sparse tensors arising from a uniform hypergraph.
In Section 4, we propose the new optimization algorithm
based on L-BFGS and the Cayley transform.
The convergence analysis of this algorithm is established in Section 5.
Numerical experiments reported in Section 6 show that the new
algorithm is efficient and promising. Finally, we conclude this paper in Section 7.

\section{Preliminary on spectral hypergraph theory}

We introduce the definitions of eigenvalues and spectral of a symmetric tensor
and then discuss developments in spectral hypergraph theory.

The conceptions of eigenvalues and associated eigenvectors of a symmetric tensor are
established by Qi \cite{Qi-05} and Lim \cite {Lim-05} independently.
Suppose
\begin{equation*}
    \Ten = (t_{i_1 \cdots i_k})\in\REAL^{[k,n]},
    \qquad \text{ for }~ i_j=1,\ldots,n, j=1,\ldots,k,
\end{equation*}
is a $k$th order $n$ dimensional symmetric tensor.
Here, the symmetry means that the value of $t_{i_1 \cdots i_k}$
is invariable under any permutation of its indices.
For $\x\in\REAL^n$, we define a scalar
\begin{equation*}
    \Ten\x^k \equiv \sum_{i_1=1}^n\cdots\sum_{i_k=1}^n
      t_{i_1\cdots i_k}x_{i_1}\cdots x_{i_k} \in\REAL.
\end{equation*}
Two column vectors $\Ten\x^{k-1}\in\REAL^n$
and $\x^{[k-1]}\in\REAL^n$ are defined with elements
\begin{equation*}
    (\Ten\x^{k-1})_i \equiv \sum_{i_2=1}^n\cdots\sum_{i_k=1}^n
      t_{ii_2\cdots i_k}x_{i_2}\cdots x_{i_k}
\end{equation*}
and $(\x^{[k-1]})_i \equiv x_i^{k-1}$ for $i=1,\ldots,n$ respectively.
Obviously, $\Ten\x^k = \x^\T(\Ten\x^{k-1})$.

If there exist a real $\lambda$ and a nonzero vector $\x\in\REAL^n$ satisfying
\begin{equation}\label{H-eig-def}
    \Ten\x^{k-1} = \lambda\x^{[k-1]},
\end{equation}
we call $\lambda$ an H-eigenvalue of $\Ten$
and $\x$ its associated H-eigenvector.
If the following system\footnote{%
  Qi \cite{Qi-05} pointed out that the tensor $\Ten$ should be regular, i.e.,
  zero is the unique solution of $\Ten\x^{k-1}=\vt{0}$.}
\begin{equation}\label{Z-eig-def}
\left\{\begin{aligned}
    \Ten\x^{k-1} &= \lambda\x, \\
    \x^{\T}\x &= 1,
\end{aligned}\right.
\end{equation}
has a real solution $(\lambda,\x)$,
$\lambda$ is named a Z-eigenvalue of $\Ten$ and
$\x$ is its associated Z-eigenvector.
All of the H- and Z-eigenvalues of $\Ten$ are called
its H-spectrum $\mathrm{Hspec}(\Ten)$ and
Z-spectrum $\mathrm{Zspec}(\Ten)$ respectively.


These definitions on eigenvalues of a symmetric tensor have
important applications in spectral hypergraph theory.

\begin{Definition}[Hypergraph]
  We denote a hypergraph by $G=(V,E)$, where $V=\{1,2,\ldots,n\}$ is the vertex set,
  $E=\{e_1,e_2,\ldots,e_m\}$ is the edge set, $e_p\subset V$ for $p=1,\ldots,m$.
  If $|e_p|=k \geq 2$ for $p=1,\ldots,m$ and $e_p \neq e_q$ in case of $p \neq q$,
  then $G$ is called a uniform hypergraph or a $k$-graph.
  If $k=2$, $G$ is an ordinary graph.

  The $k$-graph $G=(V,E)$ is called odd-bipartite if $k$ is even and there exists
  a proper subset $U$ of $V$ such that $|e_p\cap U|$ is odd for $p=1,\ldots,m$.
\end{Definition}

Let $G=(V,E)$ be a $k$-graph. For each $i\in V$, its degree $d(i)$ is defined as
\begin{equation*}
    d(i) = \left|\{e_p : i \in e_p \in E\}\right|.
\end{equation*}
We assume that every vertex has at least one edge. Thus, $d(i)>0$ for all $i$.
Furthermore, we define $\Delta$ as the maximum degree of $G$, i.e.,
$\Delta = \max_{1\leq i\leq n} d(i).$

\begin{figure}[!tbh]
  \centering
  \includegraphics[width=.4\textwidth]{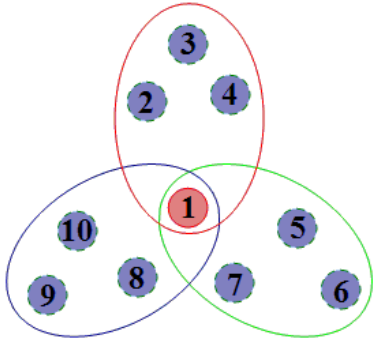}\\
  \caption{A $4$-uniform hypergraph: sunflower.}\label{HyperGraph-1}
\end{figure}

The first hypergraph is illustrated in Figure \ref{HyperGraph-1}.
There are ten vertices $V=\{1,2,\ldots,10\}$ and three edges
$E = \{e_1=\{1,2,3,4\}, e_2=\{1,5,6,7\}, e_3=\{1,8,9,10\}\}$.
Hence, it is a $4$-graph, and its degrees are
$d(1)=3$ and $d(i)=1$ for $i=2,\ldots,10$.
So we have $\Delta=3$. Moreover, this hypergraph is odd-bipartite
since we could take $U=\{1\}\subset V$.

\begin{Definition}[Adjacency tensor \cite{CoD-12}]
  Let $G=(V,E)$ be a $k$-graph with $n$ vertices.
  The adjacency tensor $\Adj=(a_{i_1 \cdots i_k})$ of $G$ is
  a $k$th order $n$-dimensional symmetric tensor, whose elements are
  \begin{equation*}
    a_{i_1 \cdots i_k}=\left\{\begin{aligned}
      &\frac{1}{(k-1)!} && \quad\text{ if }\{i_1,\ldots,i_k\}\in E, \\
      &0                && \quad\text{ otherwise. }
    \end{aligned}\right.
  \end{equation*}
\end{Definition}

\begin{Definition}[Laplacian tensor and signless Laplacian tensor \cite{Qi-14}]
  Let $G$ be a $k$-graph with $n$ vertices.
  We denote its degree tensor $\Dag$ as a $k$th order $n$-dimensional diagonal tensor
  whose $i$th diagonal element is $d(i)$.
  Then, the Laplacian tensor $\Lap$ and the signless Laplacian tensor $\sLp$ of $G$
  is defined respectively as
  \begin{equation*}
    \Lap = \Dag - \Adj \qquad\text{ and }\qquad \sLp = \Dag + \Adj.
  \end{equation*}
\end{Definition}

Obviously, the adjacency tensor $\Adj$ and the signless Laplacian tensor $\sLp$ of
a hypergraph $G$ are nonnegative. Moreover,
they are weakly irreducible if and only if $G$ is connected \cite{PeZ-14}.
Hence, we could apply the Ng-Qi-Zhou algorithms \cite{NQZ-09, CPZ-11, ZQW-13b}
for computing their largest H-eigenvalues and associated H-eigenvectors.
On the other hand, the Laplacian tensor $\Lap$ of a uniform hypergraph $G$
is a $M$-tensor \cite{ZQZ-14,DQW-13}.
Qi \cite[Theorem 3.2]{Qi-14} proved that zero is the smallest H-eigenvalue of $\Lap$.
However, the following problems are still open.
\begin{itemize}
  \item How to compute the largest H-eigenvalue of $\Lap$?
  \item How to calculate the smallest H-eigenvalues of $\sLp$ and $\Adj$?
  \item How to obtain extreme Z-eigenvalues of $\Adj$, $\Lap$, and $\sLp$?
\end{itemize}

Many theorems in spectral hypergraph theory are proved to
address H- and Z-eigenvalues of $\Adj$, $\Lap$, and $\sLp$
when the involved hypergraph has well geometric structures.
For convenience, we denote the largest H-eigenvalue and the smallest H-eigenvalue of
a tensor $\Ten$ related to a hypergraph $G$ as $\lambda_{\max}^{H}(\Ten(G))$ and
$\lambda_{\min}^{H}(\Ten(G))$ respectively.
We also define similar notations for Z-eigenvalues of that tensor.

\begin{Theorem}\label{Th:Spectrum}
  Let $G$ be a connected $k$-graph. Then the following assertions are equivalent.
  \begin{description}
    \item[](i) $k$ is even and $G$ is odd-bipartite.
    \item[](ii) $\lambda_{\max}^{H}(\Lap(G)) = \lambda_{\max}^{H}(\sLp(G))$
      (from Theorem 5.9 of \cite{HQX-15}).
    \item[](iii) $\mathrm{Hspec}(\Lap(G)) = \mathrm{Hspec}(\sLp(G))$
      (from Theorem 2.2 of \cite{SSW-15}).
    \item[](iv) $\mathrm{Hspec}(\Adj(G)) = -\mathrm{Hspec}(\Adj(G))$
      (from Theorem 2.3 of \cite{SSW-15}).
    \item[](v) $\mathrm{Zspec}(\Lap(G)) = \mathrm{Zspec}(\sLp(G))$
      (from Theorem 8 of \cite{BFZ-15}).
  \end{description}
\end{Theorem}

%
%

Khan and Fan \cite{KhF-15} studied a sort of non-odd-bipartite hypergraph
and gave the following result.

\begin{Theorem}\label{PowerG}(Corollary 3.6 of \cite{KhF-15})
  Let $G$ be a simple graph. For any positive integer $k$,
  we blow up each vertex of $G$ into a set that includes $k$ vertices
  and get a $2k$-graph $G^{2k,k}$.
  Then, $G^{2k,k}$ is not odd-bipartite if and only if $G$ is non-bipartite.
  Furthermore,
  \begin{equation*}
    \lambda_{\min}(\Adj(G)) = \lambda_{\min}^{H}(\Adj(G^{2k,k}))
    \quad\text{ and }\quad
    \lambda_{\min}(\sLp(G)) = \lambda_{\min}^{H}(\sLp(G^{2k,k})).
  \end{equation*}
\end{Theorem}

%

\section{Computational methods on sparse tensors
  arising from a hypergraph}\label{CompIssue}

The adjacency tensor $\Adj$, the Laplacian tensor $\Lap$,
and the signless Laplacian tensor $\sLp$ of a uniform hypergraph are usually sparse.
For instance, $\Adj$, $\Lap$ and $\sLp$ of
the $4$-uniform sunflower illustrated in Figure \ref{HyperGraph-1}
only contain $0.72\%$, $0.76\%$, and $0.76\%$ nonzero elements respectively.
Hence, it is an important issue to explore the sparsity in
tensors $\Adj$, $\Lap$, and $\sLp$ involved in a hypergraph $G$.
Now, we introduce a fast numerical approach based on MATLAB.

\smallskip\indent
\textbf{How to store a uniform hypergraph?}
Let $G=(V,E)$ be a $k$-graph with $|V|=n$ vertices and $|E|=m$ edges.
We store $G$ as an $m$-by-$k$ matrix $G_r$
whose rows are composed of the indices of vertices from corresponding edges of $G$.
Here, the ordering of elements in each row of $G_r$ is unimportant in the sense that
we could permute them.


For instance, we consider the $4$-uniform sunflower shown in Figure \ref{HyperGraph-1}.
The edge-vertex incidence matrix of this sunflower is a $3$-by-$10$ sparse matrix
\begin{equation*}
  \begin{array}{@{}r@{}c@{}c@{}c@{}c@{}c@{}c@{}c@{}c@{}c@{}c@{}l@{}}
    \left.\begin{array}{c} \\ \\ \\ \end{array}\right[
      & \begin{array}{c} 1 \\ 1 \\ 1 \end{array}
        & \begin{array}{c} 1 \\ 0 \\ 0 \end{array}
          & \begin{array}{c} 1 \\ 0 \\ 0 \end{array}
            & \begin{array}{c} 1 \\ 0 \\ 0 \end{array}
              & \begin{array}{c} 0 \\ 1 \\ 0 \end{array}
                & \begin{array}{c} 0 \\ 1 \\ 0 \end{array}
                  & \begin{array}{c} 0 \\ 1 \\ 0 \end{array}
                    & \begin{array}{c} 0 \\ 0 \\ 1 \end{array}
                      & \begin{array}{c} 0 \\ 0 \\ 1 \end{array}
                        & \begin{array}{c} 0 \\ 0 \\ 1 \end{array}
                          & \left]\begin{array}{c} \\ \\ \\ \end{array}\right. \\
    & \uparrow & \uparrow & \uparrow & \uparrow & \uparrow & \uparrow & \uparrow & \uparrow & \uparrow & \uparrow & \\
    & \text{\footnotesize 1} & \text{\footnotesize 2} & \text{\footnotesize 3} & \text{\footnotesize 4}
      & \text{\footnotesize 5} & \text{\footnotesize 6} & \text{\footnotesize 7} & \text{\footnotesize 8}
        & \text{\footnotesize 9} & \text{\footnotesize 10} & ~~\gets\text{\footnotesize (the indices of vertices)} \\
  \end{array}
\end{equation*}
From the viewpoint of scientific computing,
we prefer to store the incidence matrix of the sunflower in a compact form
\begin{equation*}
    G_r = \left[
       \begin{array}{cccccccccc}
         1 & 2 & 3 &  4 \\
         1 & 5 & 6 &  7 \\
         1 & 8 & 9 & 10 \\
       \end{array}
     \right] \in \REAL^{3 \times 4}.
\end{equation*}
Obviously, the number of columns of the matrix $G_r$ is less than the original incidence matrix,
since usually $k \ll n$.
We can benefit from this compact matrix in the process of computing.
In MATLAB, this matrix $G_r$ is written in Line 2 of Figure \ref{Exp-Matlab}.

\begin{figure}
\lstset{numbers=left, numberstyle=\tiny, keywordstyle=\color{blue!70}, commentstyle=\color{red!50!green!50!blue!50}, frame=shadowbox, rulesepcolor=\color{red!20!green!20!blue!20},escapeinside=``, xleftmargin=2em,xrightmargin=2em, aboveskip=1em}
\begin{lstlisting}[language=Matlab]
  % Store a 4-uniform sunflower
  Gr = [1,2,3,4; 1,5,6,7; 1,8,9,10];

  % Calculate the degree vector
  [m,k] = size(Gr);  n = max(Gr(:));
  Msp = sparse(Gr(:),(1:m*k)',ones(m*k,1),n,m*k);
  degree = full(sum(Msp,2));

  % Suppose that x is an n-dimensional vector.
  % Compute the vector Dx^(k-1)
  Dx_ = degree .* (x.^(k-1));

  % Compute the scalar Dx^k
  Dxk  = dot(degree, x.^k); % direct approach, or
  Dxk2 = dot(x, Dx_);       % if Dx_ is available.

  % Compute the vector Ax^(k-1)
  Xmat = reshape(x(Gr(:)),[m,k]);
  Ax_  = zeros(n,1);
  for j=1:k
      Mj  = sparse(Gr(:,j),(1:m)',ones(m,1),n,m);
      yj  = prod(Xmat(:,[1:j-1,j+1:k]),2);
      Ax_ = Ax_ + Mj*yj;
  end

  % Compute the scalar Ax^k
  Axk  = k*sum(prod(Xmat,2)); % direct approach, or
  Axk2 = dot(x, Ax_);         % if Ax_ is available.

  % Calculate Lx^(k-1) and Qx^(k-1)
  Lx_ = Dx_ - Ax_;
  Qx_ = Dx_ + Ax_;

  % Calculate Lx^k and Qx^k
  Lxk = Dxk - Axk;
  Qxk = Dxk + Axk;
\end{lstlisting}
\caption{Matlab codes for the sunflower illustrated in
  Figure \ref{HyperGraph-1}.}\label{Exp-Matlab}
\end{figure}

\smallskip\indent
\textbf{How to compute products $\Ten\x^k$ and $\Ten\x^{k-1}$
  when $\Ten=\Adj,\Lap,\text{ and }\sLp$?}
Suppose that the matrix $G_r$ representing a uniform hypergraph
and a vector $\x\in\REAL^n$ are available.
Since $\Lap = \Dag-\Adj$ and $\sLp = \Dag+\Adj$,
it is sufficient to study the degree tensor $\Dag$ and the adjacency tensor $\Adj$.

We first consider the degree tensor $\Dag$. It is a diagonal tensor and
its $i$th diagonal element is the degree $d(i)$ of a vertex $i\in V$.
Once the hypergraph $G$ is given, the degree vector $\dvec \equiv [d(i)]\in\REAL^n$ is fixed.
So we could save $\dvec$ from the start.
Let $\delta(\cdot,\cdot)$ be the Kronecker delta, i.e.,
$\delta(i,j)=1$ if $i=j$ and $\delta(i,j)=0$ if $i\neq j$.
Using this notation, we could rewrite the degree as
\begin{equation*}
    d(i) = \sum_{\ell=1}^m\sum_{j=1}^k\delta(i,(G_r)_{\ell j}),
    \qquad \text{ for }~ i=1,\ldots,n.
\end{equation*}
To calculate the degree vector $d$ efficiently,
we construct an $n$-by-$mk$ sparse matrix $M_{sp}=[\delta(i,(G_r)_{\ell j})]$.
By summarizing each row of $M_{sp}$, we obtain the degree vector $\dvec$.
For any vector $\x\in\REAL^n$, the computation of
\begin{equation*}
    \Dag\x^{k-1}= \dvec \ast (\x^{[k-1]}) \qquad\text{and}\qquad
    \Dag\x^k    = \dvec^\T(\x^{[k]})
\end{equation*}
are straightforward, where ``$\ast$'' denotes the component-wise Hadamard product.
In Figure \ref{Exp-Matlab}, we show these codes in Lines 4-15.

Second, we focus on the adjacency tensor $\Adj$.
We construct a matrix $X_{mat}=[x_{(G_r)_{\ell j}}]$ which has the same size as $G_r$.
Assume that the $(\ell,j)$-th element of $G_r$ is $i$. Then,
the $(\ell,j)$-th element of $X_{mat}$ is defined as $x_i$.
From this matrix, we rewrite the product $\Adj\x^k$ as
\begin{equation*}
    \Adj\x^k = k \sum_{\ell=1}^m \prod_{j=1}^k (X_{mat})_{\ell j}.
\end{equation*}
See Lines 18 and 27 of Figure \ref{Exp-Matlab}.
To compute the vector $\Adj\x^{k-1}$, we use the following representation
\begin{equation*}
    (\Adj\x^{k-1})_i = \sum_{j=1}^k \sum_{\ell=1}^m
      \left(\delta(i,(G_r)_{\ell j})\prod_{\substack{s=1 \\ s\neq j}}^k (X_{mat})_{\ell s}\right),
    \qquad \text{ for }~ i=1,\ldots,n.
\end{equation*}
For each $j=1,\ldots,k$, we construct
a sparse matrix $M_j=[\delta(i,(G_r)_{\ell j})]\in\REAL^{n \times m}$
and a column vector $\y_j=[\prod_{s\neq j} (X_{mat})_{\ell s}]\in\REAL^m$
respectively. Then, the vector $$\Adj\x^{k-1} = \sum_{j=1}^k M_j\y_j$$
could be computed by using a simple loop.
See Lines 17-24 of Figure \ref{Exp-Matlab}.

The computational costs for computing products of tensors $\Adj$, $\Lap$, and $\sLp$
with any vector $\x$ are about $mk^2$, $mk^2+nk$, and
$mk^2+nk$ multiplications, respectively.
Since $mk^2 < mk^2+nk \leq 2mk^2$, the computational cost of the product of
a vector and a large scale sparse tensor related with a hypergraph is cheap.
Additionally, the codes listed in Figure \ref{Exp-Matlab}
could easily be extended to parallel computing.


\section{The CEST algorithm}
The design of the novel CEST algorithm is based on a unified formula for
the H- and Z-eigenvalue of a symmetric tensor \cite{CPZ-09,DiW-15}.
Let $\Iid\in\REAL^{[k,n]}$ be an identity tensor
whose diagonal elements are all one and off-diagonal elements are zero.
Hence, $\Iid\x^{k-1}=\x^{[k-1]}$.
If $k$ is even, we define $\Eid\in\REAL^{[k,n]}$ such that
$ \Eid\x^{k-1} = (\x^\T\x)^{\frac{k}{2}-1}\x. $
Using tensors $\Iid$ and $\Eid$, we could rewrite
systems \eqref{H-eig-def} and \eqref{Z-eig-def} as
\begin{equation}\label{geig}
    \Ten\x^{k-1}=\lambda\Btn\x^{k-1},
\end{equation}
where $\Btn=\Iid$ and $\Btn=\Eid$ respectively.
In the remainder of this paper, we call a real $\lambda$ and
a nonzero vector $\x\in\REAL^n$ an eigenvalue and its associated eigenvector respectively
if they satisfies \eqref{geig}.
Now, we devote to compute such $\lambda$ and $\x$
for large scale sparse tensors.

Let $k$ be even. We consider the spherical optimization problem
\begin{equation}\label{Sph-Opt}
    \min~f(\x) = \frac{\Ten\x^k}{\Btn\x^k} \qquad
    \st~~\x\in\SPHERE,
\end{equation}
where the symmetric tensor $\Ten$ arises from a $k$-uniform hypergraph,
so $\Ten$ is sparse and may be large scale.
$\Btn$ is a symmetric positive definite tensor with a simple structure
such as $\Iid$ and $\Eid$.
Without loss of generality, we restrict $\x$ on
a compact unit sphere $\SPHERE \equiv \{\x\in\REAL^n : \x^\T\x=1\}$
because $f(\x)$ is zero-order homogeneous.

The gradient of $f(\x)$ \cite{CQW-15} is
\begin{equation}\label{Opt-grad}
    \g(\x) = \frac{k}{\Btn\x^k}\left(\Ten\x^{k-1}-\frac{\Ten\x^k}{\Btn\x^k}\Btn\x^{k-1}\right).
\end{equation}
Obviously, for all $\x\in\SPHERE$, we have
\begin{equation}\label{xkgk}
    \x^\T\g(\x) = \frac{k}{\Btn\x^k}\left(\x^\T\Ten\x^{k-1}-\frac{\Ten\x^k}{\Btn\x^k}\x^\T\Btn\x^{k-1}\right) = 0.
\end{equation}
This equality implies that the vector $\x\in\SPHERE$ is perpendicular to
its (negative) gradient direction.
The following theorem reveals the relationship between the spherical optimization
\eqref{Sph-Opt} and the eigenvalue problem \eqref{geig}.

\begin{Theorem}\label{Th: opt-cond}
  Suppose that the order $k$ is even and the symmetric tensor $\Btn$ is positive definite.
  Let $\x_*\in\SPHERE$.
  Then, $\x_*$ is a first-order stationary point, i.e., $\g(\x_*)=0$,
  if and only if $\x_*$ is an eigenvector corresponding to a certain eigenvalue.
  In fact, the eigenvalue is $f(\x_*)$.
\end{Theorem}
\begin{proof}
  Since $\Btn$ is positive definite, $\Btn\x^k>0$ for all $\x\in\SPHERE$.
  Hence, by \eqref{Opt-grad}, if $\x_*\in\SPHERE$ satisfies $\g(\x_*)=0$,
  $f(\x_*)$ is an eigenvalue and $\x_*$ is its associated eigenvector.

  On the other hand, suppose that $\x_*\in\SPHERE$ is an eigenvector
  corresponding to an eigenvalue $\lambda_*$, i.e.,
  \begin{equation*}
    \Ten\x_*^{k-1} = \lambda_*\Btn\x_*^{k-1}.
  \end{equation*}
  By taking inner products on both sides with $\x_*$,
  we get $\Ten\x_*^k = \lambda_*\Btn\x_*^k$.
  Because $\Btn\x_*^k>0$, it yields that $\lambda_*=\frac{\Ten\x_*^k}{\Btn\x_*^k}=f(\x_*)$.
  Hence, by \eqref{Opt-grad}, we obtain $\g(\x_*)=0$.
\end{proof}

Next, we focus on numerical approaches for computing a first-order stationary
point of the spherical optimization \eqref{Sph-Opt}.
First, we apply the limited memory BFGS (L-BFGS) approach for generating a search direction.
Then, a curvilinear search technique is explored to preserve iterates
in a spherical constraint.

\subsection{L-BFGS produces a search direction}

The limited memory BFGS method is powerful for large scale nonlinear unconstrained optimization.
In the current iteration $c$,
it constructs an implicit matrix $H_c$ to approximate the inverse of a Hessian of $f(\x)$.
At the beginning, we introduce the basic BFGS update.


BFGS is a quasi-Newton method which updates
the approximation of the inverse of a Hessian iteratively.
Let $H_c$ be the current approximation,
\begin{equation}\label{def-sy}
    \y_c = \g(\x_{c+1})-\g(\x_c), \qquad
    \s_c = \x_{c+1}-\x_c, \quad\text{and}\quad
     V_c = I-\rho_c\y_c\s_c^\T,
\end{equation}
where $I$ is an identity matrix,
\begin{equation}\label{def-rho}
    \rho_c = \left\{\begin{aligned}
      &\frac{1}{\y_c^\T\s_c} &&~ \text{ if }\y_c^\T\s_c \geq \kappa_{\epsilon}, \\
      &0                     &&~ \text{ otherwise, }
    \end{aligned}\right.
\end{equation}
and $\kappa_{\epsilon}\in(0,1)$ is a small positive constant.
We generate the new approximation $H_c^+$ by the BFGS formula \cite{NoW-06,SuY-06}
\begin{equation}\label{BFGS-formula}
    H_{c}^+ = V_c^\T H_c V_c + \rho_c\s_c\s_c^\T.
\end{equation}

For the purpose of solving large scale optimization problems,
Nocedal \cite{Noc-80} proposed the L-BFGS approach
which implements the BFGS update in an economic way.
Given any vector $\g\in\REAL^n$, the matrix-vector product $-H_c\g$ could be computed
using only $\mathcal{O}(n)$ multiplications.

In each iteration $c$, L-BFGS starts from a simple matrix
\begin{equation}\label{Hk-start}
    H_c^{(0)} = \gamma_c I,
\end{equation}
where $\gamma_c>0$ is usually determined by
the Barzilai-Borwein method \cite{LiN-89,BaB-88}.
Then, we use BFGS formula \eqref{BFGS-formula} to update $H_c^{(\ell)}$ recursively
\begin{equation}\label{Hk-rec}
    H_c^{(L-\ell+1)} = V_{c-\ell}^\T H_c^{(L-\ell)} V_{c-\ell}
      + \rho_{c-\ell}\s_{c-\ell}\s_{c-\ell}^\T, \qquad
      \text{for }~ \ell = L,L-1,\ldots,1,
\end{equation}
and obtain
\begin{equation}\label{Hk-final}
    H_c=H_c^{(L)}.
\end{equation}
If $\ell\geq c$, we define $\rho_{c-\ell}=0$ and L-BFGS does nothing for that $\ell$.
In a practical implementation, L-BFGS enjoys a cheap two-loop recursion.
The computational cost is about $4Ln$ multiplications.

\renewcommand{\thealgorithm}{L-BFGS}
\begin{algorithm}[!tb]
\newcommand{\qk}{\vt{q}}
\newcommand{\rk}{\vt{p}}
\caption{The two-loop recursion for L-BFGS \cite{Noc-80,NoW-06,SuY-06}.}\label{L-BFGS 2-loop}
\begin{algorithmic}[1]
  \STATE $\qk \gets -\g(\x_c)$,
  \FOR{$i=c-1,c-2,\ldots,c-L$}
      \STATE $\alpha_i \gets \rho_i\s_i^\T\qk$,
      \STATE $\qk \gets \qk-\alpha_i\y_i$,
  \ENDFOR
  \STATE $\rk \gets \gamma_c \qk$,
  \FOR{$i=c-L,c-L+1,\ldots,c-1$}
      \STATE $\beta \gets \rho_i\y_i^\T\rk$,
      \STATE $\rk \gets \rk + \s_i(\alpha_i-\beta)$,
  \ENDFOR
  \STATE Stop with result $\rk = -H_c\g(\x_c)$.
\end{algorithmic}
\end{algorithm}

For the parameter $\gamma_c$, we have three candidates.
The first two are suggested by Barzilai and Borwein \cite{BaB-88}:
\begin{equation}\label{BB-steps}
    \gamma_c^{\mathrm{BB1}} = \frac{\y_c^\T\s_c}{\|\y_c\|^2}
    \qquad\text{ and }\qquad
    \gamma_c^{\mathrm{BB2}} = \frac{\|\s_c\|^2}{\y_c^\T\s_c}.
\end{equation}
The third one are their geometric mean \cite{Dai-14}
\begin{equation}\label{Dai-steps}
    \gamma_c^{\mathrm{Dai}} = \frac{\|\s_c\|}{\|\y_c\|}.
\end{equation}
Furthermore, we set $\gamma_c=1$ if $\y_c^\T\s_c < \kappa_{\epsilon}$.

\subsection{Cayley transform preserves the spherical constraint}

Suppose that $\x_c\in\SPHERE$ is the current iterate,
$\p_c\in\REAL^n$ is a good search direction generated by
Algorithm \ref{L-BFGS 2-loop} and $\alpha$ is a damped factor.
First, we construct a skey-symmetric matrix
\begin{equation}\label{W-mat}
    W = \alpha(\x_c\p_c^\T - \p_c\x_c^\T) \in\REAL^{n\times n}.
\end{equation}
Obviously, $I+W$ is invertible.
Using the Cayley transform, we obtain an orthogonal matrix
\begin{equation}\label{Q-mat}
    Q = (I-W)(I+W)^{-1}.
\end{equation}
Hence, the new iterate $\x_{c+1}$ is still locating on the unit sphere $\SPHERE$
if we define
\begin{equation}\label{new-it}
    \x_{c+1} = Q\x_c.
\end{equation}
Indeed, matrices $W$ and $Q$ are not needed to be formed explicitly.
The new iterate $\x_{c+1}$ could be generated from $\x_c$ and $\p_c$ directly
with only about $4n$ multiplications.

\begin{Lemma}\label{Lem:xkiter}
  Suppose that the new iterate $\x_{c+1}$ is generated by \eqref{W-mat}, \eqref{Q-mat},
  and \eqref{new-it}. Then, we have
  \begin{equation}\label{iterate-update}
    \x_{c+1}(\alpha) = \frac{[(1-\alpha\x_c^\T\p_c)^2-\|\alpha\p_c\|^2]\x_c+2\alpha\p_c}{
      1+\|\alpha\p_c\|^2-(\alpha\x_c^\T\p_c)^2}
  \end{equation}
  and
  \begin{equation}\label{iterate-length}
    \|\x_{c+1}(\alpha)-\x_c\| =
      2\left(\frac{\|\alpha\p_c\|^2-(\alpha\x_c^\T\p_c)^2}{
            1+\|\alpha\p_c\|^2-(\alpha\x_c^\T\p_c)^2}\right)^{\frac{1}{2}}.
  \end{equation}
\end{Lemma}
\begin{proof}
  We employ the Sherman-Morrison-Woodbury formula: if $A$ is invertible,
  \begin{equation*}
    (A+UV^\T)^{-1} = A^{-1}-A^{-1}U(I+{V^\T}A^{-1}U)^{-1}{V^\T}A^{-1}.
  \end{equation*}
  It yields that
  \begin{eqnarray*}
    \lefteqn{ (I+W)^{-1}\x_c = \left(I+
      \left[\begin{array}{cc} \x_c & -\alpha\p_c \\      \end{array}\right]
      \left[\begin{array}{c}  \alpha\p_c^\T \\ \x_c^\T \\ \end{array}\right]
    \right)^{-1}\x_c }\\
      &=& \left(I-\left[\begin{array}{cc}\x_c & -\alpha\p_c \\\end{array}\right]
        \left(\left[\begin{array}{cc} 1 & 0 \\0 & 1 \\ \end{array}\right]+
        \left[\begin{array}{c}  \alpha\p_c^\T \\ \x_c^\T \\ \end{array}\right]
        \left[\begin{array}{cc} \x_c & -\alpha\p_c \\       \end{array}\right] \right)^{-1}
        \left[\begin{array}{c} \alpha\p_c^\T \\ \x_c^\T \\  \end{array}\right]\right)\x_c \\
      &=& \x_c-\left[\begin{array}{cc}\x_c & -\alpha\p_c \\ \end{array}\right]
        \left[\begin{array}{cc}
                1+\alpha\x_c^\T\p_c & -\|\alpha\p_c\|^2 \\
                1 & 1-\alpha\x_c^\T\p_c \\
              \end{array}\right]^{-1}
        \left[\begin{array}{c}\alpha\x_c^\T\p_c \\ 1 \\ \end{array}\right] \\
      &=& \x_c-\left[\begin{array}{cc}\x_c & -\alpha\p_c \\ \end{array}\right]
            \frac{1}{1+\|\alpha\p_c\|^2-(\alpha\x_c^\T\p_c)^2}
        \left[\begin{array}{c}
                \alpha\x_c^\T\p_c(1-\alpha\x_c^\T\p_c)+\|\alpha\p_c\|^2 \\ 1 \\
        \end{array}\right] \\
      &=& \frac{(1-\alpha\x_c^\T\p_c)\x_c+\alpha\p_c}{1+\|\alpha\p_c\|^2-(\alpha\x_c^\T\p_c)^2},
  \end{eqnarray*}
  where $1+\|\alpha\p_c\|^2-(\alpha\x_c^\T\p_c)^2 \geq 1$ since $\x_c\in\SPHERE$
  and $|\alpha\x_c^\T\p_c| \leq \|\alpha\p_c\|$.
  Then, the calculation of $\x_{c+1}$ is straightforward
  \begin{equation*}
    \x_{c+1} = (I-W)\frac{(1-\alpha\x_c^\T\p_c)\x_c+\alpha\p_c}{1+\|\alpha\p_c\|^2-(\alpha\x_c^\T\p_c)^2}
      = \frac{[(1-\alpha\x_c^\T\p_c)^2-\|\alpha\p_c\|^2]\x_c+2\alpha\p_c}{1+\|\alpha\p_c\|^2-(\alpha\x_c^\T\p_c)^2}.
  \end{equation*}
  Hence, the iterate formula \eqref{iterate-update} is valid.

  Then, by some calculations, we have
  \begin{eqnarray*}
    \lefteqn{ \|\x_{c+1}(\alpha)-\x_c\|^2 }\\
      &=& \left\|\frac{[2\alpha\x_c^\T\p_c(\alpha\x_c^\T\p_c-1)-2\|\alpha\p_c\|^2]\x_c+2\alpha\p_c}{
            1+\|\alpha\p_c\|^2-(\alpha\x_c^\T\p_c)^2}\right\|^2 \\
      &=& \frac{[2\alpha\x_c^\T\p_c(\alpha\x_c^\T\p_c-1)-2\|\alpha\p_c\|^2]
            [2\alpha\x_c^\T\p_c(\alpha\x_c^\T\p_c+1)-2\|\alpha\p_c\|^2]+ 4\|\alpha\p_c\|^2}{
            [1+\|\alpha\p_c\|^2-(\alpha\x_c^\T\p_c)^2]^2} \\
      &=& \frac{(2\alpha\x_c^\T\p_c)^2[(\alpha\x_c^\T\p_c)^2-1]
            -2\|\alpha\p_c\|^2(2\alpha\x_c^\T\p_c)^2+4\|\alpha\p_c\|^4+ 4\|\alpha\p_c\|^2}{
            [1+\|\alpha\p_c\|^2-(\alpha\x_c^\T\p_c)^2]^2} \\
      &=& \frac{(2\alpha\x_c^\T\p_c)^2[(\alpha\x_c^\T\p_c)^2-1-\|\alpha\p_c\|^2]
            +4\|\alpha\p_c\|^2[-(\alpha\x_c^\T\p_c)^2+\|\alpha\p_c\|^2+1]}{
            [1+\|\alpha\p_c\|^2-(\alpha\x_c^\T\p_c)^2]^2} \\
      &=& \frac{4\|\alpha\p_c\|^2-4(\alpha\x_c^\T\p_c)^2}{
            1+\|\alpha\p_c\|^2-(\alpha\x_c^\T\p_c)^2}.
  \end{eqnarray*}
  Therefore, the equality \eqref{iterate-length} holds.
\end{proof}

Whereafter, the damped factor $\alpha$ could be determined by an inexact line earch
owing to the following theorem.

\begin{Theorem}\label{Th: Step-size}
  Suppose that $\p_c$ is a gradient-related direction satisfying \eqref{grad-relat}
  and $\x_{c+1}(\alpha)$ is generated by (\ref{iterate-update}).
  Let $\eta\in(0,1)$ and $\g(\x_c)\neq 0$.
  Then, there is an $\tilde{\alpha}_c>0$ such that
  for all $\alpha\in(0,\tilde{\alpha}_c]$,
  \begin{equation}\label{suf-dec}
    f(\x_{c+1}(\alpha)) \leq f(\x_c) + \eta\alpha\p_c^\T\g(\x_c).
  \end{equation}
\end{Theorem}
\begin{proof}
  From \eqref{iterate-update}, we obtain
  $\x_{c+1}(0)=\x_c$ and $\x'_{c+1}(0)=-2\x_c^\T\p_c\x_c+2\p_c$.
  Hence, we have
  \begin{equation*}
    \left.\frac{\diff f(\x_{c+1}(\alpha))}{\diff \alpha}\right|_{\alpha=0}
      = \g(\x_{c+1}(0))^\T\x'_{c+1}(0)
      = \g(\x_c)^\T(-2\x_c^\T\p_c\x_c+2\p_c)
      = 2\p_c^\T\g(\x_c),
  \end{equation*}
  where the last equality holds for \eqref{xkgk}.
  Since $\g(\x_c)\neq 0$ and $\p_c$ satisfies \eqref{grad-relat},
  we have $\p_c^\T\g(\x_c)<0$.
  Then, by Taylor's theorem, for a sufficiently small $\alpha$, we obtain
  \begin{equation*}
    f(\x_{c+1}(\alpha)) = f(\x_c) + 2\alpha\p_c^\T\g(\x_c) + o(\alpha^2).
  \end{equation*}
  Owing to $\eta<2$, there exists a positive $\tilde{\alpha}_c$ such that
  \eqref{suf-dec} is valid.
\end{proof}

Finally, we present the new Algorithm \ref{alg} formally.
Roughly speaking, CEST is a modified version of the state-of-the-art L-BFGS method
for unconstrained optimization. Due to the spherical constraint imposed here,
we use the Cayley transform explicitly to preserve iterates on a unit sphere.
An inexact line search is employed to determine a suitable damped factor.
Theorem \ref{Th: Step-size} indicates that the inexact line search is well-defined.

\renewcommand{\thealgorithm}{CEST}
\begin{algorithm}[tb!]
\caption{Computing eigenvalues of sparse tensors.}\label{alg}
\begin{algorithmic}[1]
  \STATE For a given uniform hypergraph $G_r$, we compute the degree vector $\dvec$.

  \STATE Choose an initial unit iterate $\x_1$,
    a positive integer $L$, parameters $\eta \in (0,1)$,
    $\beta\in(0,1)$, and $c \gets 1$.

  \WHILE{the sequence of iterates does not converge}

    \STATE Compute $\Ten\x_c^{k-1}$ and $\Ten\x_c^k$ using the codes in Figure \ref{Exp-Matlab},
             where $\Ten \in\{\Adj,\Lap,\sLp\}$.

    \STATE Calculate $\lambda_c=f(\x_c)$ and $\g(\x_c)$ by
             \eqref{Sph-Opt} and \eqref{Opt-grad} respectively.
    \STATE Generate $\p_c=-H_c\g(\x_c)$ by Algorithm \ref{L-BFGS 2-loop}.

    \STATE Choose the smallest nonnegative integer $\ell$ and
      calculate $\alpha=\beta^{\ell}$ such that \eqref{suf-dec} holds.

    \STATE Let $\alpha_c=\beta^{\ell}$ and
      update the new iterate $\x_{c+1}=\x_{c+1}(\alpha_c)$ by \eqref{iterate-update}.
    \STATE Compute $\s_c, \y_c$ and $\rho_c$ by \eqref{def-sy} and \eqref{def-rho} respectively.
    \STATE $c \gets c+1.$
  \ENDWHILE
\end{algorithmic}
\end{algorithm}

\section{Convergence analysis}

First, we prove that the sequence of merit function values $\{f(\x_c)\}$
converges and every accumulation point of iterates $\{\x_c\}$ is a first-order
stationary point. Second, using the Kurdyka-{\L}ojasiewicz property,
we show that the sequence of iterates $\{\x_c\}$ is also convergent.
When the second-order sufficient condition holds at the limiting point,
CEST enjoys a linear convergence rate. Finally, when we start CEST from
plenty of randomly initial points, resulting eigenvalues may touch the extreme eigenvalue
of a tensor with a high probability.

\subsection{Basic convergence theory}

If CEST terminates finitely, i.e.,
there exists an iteration $c$ such that $\g(\x_c)=0$,
we immediately know that $f(\x_c)$ is an eigenvalue and $\x_c$ is
the corresponding eigenvector by Theorem \ref{Th: opt-cond}. So,
in the remainder of this section, we assume that
CEST generates an infinite sequence of iterates $\{\x_c\}$.

Since the symmetric tensor $\Btn$ is positive definite,
the merit function $f(\x)$ is twice continuously differentiable.
Owing to the compactness of the spherical domain of $f(\x)$,
we obtain the following bounds \cite{CQW-15}.

\begin{Lemma}\label{Lem:fg-bnd}
  There exists a constant $M>1$ such that
  \begin{equation*}
    |f(\x)| \leq M, \quad \|\g(\x)\| \leq M, ~~\text{ and }~~
    \|\nabla^2f(\x)\| \leq M, \qquad \forall~\x\in\SPHERE.
  \end{equation*}
\end{Lemma}

Because the bounded sequence $\{f(\x_c)\}$ decreases monotonously, it converges.

\begin{Theorem}\label{Th:eigvalue-conv}
  Assume that CEST generates an infinite sequence of merit functions $\{f(\x_c)\}$.
  Then, there exists a $\lambda_*$ such that
  \begin{equation*}
    \lim_{c\to\infty} f(\x_c) = \lambda_*.
  \end{equation*}
\end{Theorem}

The next theorem shows that $\p_c = -H_c\g(\x_c)$ generated by L-BFGS is
a gradient-related direction.

\begin{Theorem}\label{Th: grad-dir}
  Suppose that $\p_c = -H_c\g(\x_c)$ is generated by L-BFGS. Then,
  there exist constants $0< C_L \leq 1 \leq C_U$ such that
  \begin{equation}\label{grad-relat}
    \p_c^\T\g(\x_c) \leq -C_L\|\g(\x_c)\|^2
    \qquad\text{ and }\qquad
    \|\p_c\| \leq C_U\|\g(\x_c)\|.
  \end{equation}
\end{Theorem}
\begin{proof}
  See Appendix A.
\end{proof}

Using the gradient-related direction, we establish bounds for damped factors
generated by the inexact line search.

\begin{Lemma}\label{Lem:step-min}
  There exists a constant $\alpha_{\min}>0$ such that
  \begin{equation*}
    \alpha_{\min} \leq \alpha_c \leq 1,
    \qquad \forall\,c.
  \end{equation*}
\end{Lemma}
\begin{proof}
  Let $0<\alpha\leq\hat{\alpha}\equiv \frac{(2-\eta)C_L}{(2+\eta)MC_U^2}$.
  Hence, $\alpha C_UM \leq \frac{(2-\eta)C_L}{(2+\eta)C_U} <1$.
  From \eqref{grad-relat} and Lemma \ref{Lem:fg-bnd}, we obtain
  \begin{equation*}
    -\alpha\p_c^\T\g(\x_c) \leq \alpha\|\p_c\|\|\g(\x_c)\|
      \leq \alpha C_U\|\g(\x_c)\|^2 \leq \alpha C_UM^2 < M
  \end{equation*}
  and
  \begin{equation*}
    \|\alpha\p_c\|^2-(\alpha\x_c^\T\p_c)^2 \leq \alpha^2\|\p_c\|^2
    \leq \alpha^2C_U^2\|\g(\x_c)\|^2.
  \end{equation*}
  The above two inequalities and $\alpha\in(0,\hat{\alpha}]$ yield that
  \begin{eqnarray}
    \lefteqn{2\alpha\p_c^\T\g(\x_c)+2M(\|\alpha\p_c\|^2-(\alpha\x_c^\T\p_c)^2)
     -\eta\alpha\p_c^\T\g(\x_c)(1+\|\alpha\p_c\|^2-(\alpha\x_c^\T\p_c)^2)} \nonumber\\
      &=& (2-\eta)\alpha\p_c^\T\g(\x_c) + (2M-\eta\alpha\p_c^\T\g(\x_c))(\|\alpha\p_c\|^2-(\alpha\x_c^\T\p_c)^2) \nonumber\\
      &<& (2-\eta)\alpha\p_c^\T\g(\x_c) + (2+\eta)M\alpha^2C_U^2\|\g(\x_c)\|^2 \nonumber\\
      &\leq& (2-\eta)\alpha\p_c^\T\g(\x_c) + (2-\eta)C_L\alpha\|\g(\x_c)\|^2 \nonumber\\
      &\leq& 0, \label{aaaaaaaa}
  \end{eqnarray}
  where the last inequality holds for \eqref{grad-relat}.

  From the mean value theorem, Lemmas \ref{Lem:fg-bnd} and \ref{Lem:xkiter},
  and the equality \eqref{xkgk}, we have
  \begin{eqnarray*}
    f(\x_{c+1}(\alpha))-f(\x_c)
      &\leq& \g(\x_c)^\T(\x_{c+1}(\alpha)-\x_c)+\frac{1}{2}M\|\x_{c+1}(\alpha)-\x_c\|^2 \\
      &=& \frac{2\alpha\p_c^\T\g(\x_c)+2M(\|\alpha\p_c\|^2-(\alpha\x_c^\T\p_c)^2)}{1+\|\alpha\p_c\|^2-(\alpha\x_c^\T\p_c)^2} \\
      &<& \frac{\eta\alpha\p_c^\T\g(\x_c)(1+\|\alpha\p_c\|^2-(\alpha\x_c^\T\p_c)^2)}{1+\|\alpha\p_c\|^2-(\alpha\x_c^\T\p_c)^2} \\
      &=&\eta\alpha\p_c^\T\g(\x_c),
  \end{eqnarray*}
  where the last inequality is valid owing to \eqref{aaaaaaaa}.
  Due to the rule of the inexact search, the damped factor $\alpha_c$ satisfies
  $1\geq\alpha_c\geq \beta\hat{\alpha}\equiv\alpha_{\min}$.
\end{proof}

The next theorem proves that every accumulation point of iterates $\{\x_c\}$
is a first-order stationary point.

\begin{Theorem}\label{Th:basicConv}
  Suppose that CEST generates an infinite sequence of iterates $\{\x_c\}$. Then,
  \begin{equation*}
    \lim_{c\to\infty} \|\g(\x_c)\| = 0.
  \end{equation*}
\end{Theorem}
\begin{proof}
  From \eqref{suf-dec} and \eqref{grad-relat}, we get
  \begin{equation}\label{ccc-2}
    f(\x_c)-f(\x_{c+1}) \geq -\eta \alpha_c\p_c^\T\g(\x_c)
      \geq \eta\alpha_cC_L\|\g(\x_c)\|^2.
  \end{equation}
  Since Lemmas \ref{Lem:fg-bnd} and \ref{Lem:step-min}, we have
  \begin{equation*}
    2M \geq f(\x_1)-\lambda_*
      = \sum_{c=1}^{\infty} [f(\x_c)-f(\x_{c+1})]
      \geq \sum_{c=1}^{\infty} \eta\alpha_cC_L\|\g(\x_c)\|^2
      \geq \sum_{c=1}^{\infty} \eta\alpha_{\min}C_L\|\g(\x_c)\|^2.
  \end{equation*}
  That is to say,
  \begin{equation*}
    \sum_{c=1}^{\infty} \|\g(\x_c)\|^2 \leq \frac{2M}{\eta\alpha_{\min}C_L}
      < +\infty.
  \end{equation*}
  Hence, this theorem is valid.
\end{proof}

\subsection{Convergence of the sequence of iterates}

The Kurdyka-{\L}ojasiewicz property was discovered by
S. {\L}ojasiewicz \cite{Loj-63} for real-analytic functions in 1963.
Bolte et al. \cite{BDL-07} extended this property to nonsmooth functions.
Whereafter, KL property was widely applied in
analyzing proximal algorithms for nonconvex and nonsmooth optimization
\cite{ABRS10,ABS-13,BST-14,XuY-13}.

We remark that the merit function $f(\x)=\frac{\Ten\x^k}{\Btn\x^k}$
is a semialgebraic function since its graph
\begin{equation*}
    \mathrm{Graph}f = \{(\x,\lambda)\,:\, \Ten\x^k-\lambda\Btn\x^k=0 \}
\end{equation*}
is a semialgebraic set.
Therefore, $f(\x)$ satisfies the following KL property
\cite{BDL-07, AtB-09}.

\begin{Theorem}[KL property]\label{Th4-05}
  Suppose that $\x_*$ is a stationary point of $f(\x)$.
  Then, there is a neighborhood $\mathscr{U}$ of $\x_*$,
  an exponent $\theta\in[0,1)$, and a positive constant $C_K$ such that
  for all $\x\in\mathscr{U}$, the following inequality holds
  \begin{equation}\label{KL-ineq}
    |f(\x)-f(\x_*)|^{\theta} \leq C_K\|\g(\x)\|.
  \end{equation}
  Here, we define $0^0\equiv 0$.
\end{Theorem}

Using KL property, we will prove that
the infinite sequence of iterates $\{\x_c\}$ converges to a unique accumulation point.

\begin{Lemma}\label{Lem:KL-pre}
  Suppose that $\x_*$ is a stationary point of $f(\x)$, and
  $\mathscr{B}(\x_*,\rho)=\{\x\in\REAL^n : \|\x-\x_*\|\leq\rho\} \subseteq \mathscr{U}$
  is a neighborhood of $\x_*$. Let $\x_1$ be an initial point satisfying
  \begin{equation}\label{func-rho}
    \rho > \rho(\x_1)\equiv
      \frac{2C_UC_K}{\eta C_L (1-\theta)}|f(\x_1)-f(\x_*)|^{1-\theta} +\|\x_1-\x_*\|.
  \end{equation}
  Then, the following assertions hold:
  \begin{equation}\label{Lem:KL1a}
    \x_c\in\mathscr{B}(\x_*,\rho), \qquad c=1,2,\ldots,
  \end{equation}
  and
  \begin{equation}\label{Lem:KL1b}
    \sum_{c=1}^{\infty}\|\x_{c+1}-\x_c\| \leq
      \frac{2C_U C_K}{\eta C_L (1-\theta)}|f(\x_1)-f(\x_*)|^{1-\theta}.
  \end{equation}
\end{Lemma}
\begin{proof}
  We proceed by induction.
  Obviously, $\x_1\in\mathscr{B}(\x_*,\rho)$.

  Now, we assume that $\x_i\in\mathscr{B}(\x_*,\rho)$ for all $i=1,\ldots,c$.
  Hence, KL property holds in these points.
  Let
  \begin{equation*}
    \phi(t) \equiv \frac{C_K}{1-\theta}|t-f(\x_*)|^{1-\theta}.
  \end{equation*}
  It is easy to prove that $\phi(t)$ is a concave function for $t>f(\x_*)$.
  Therefore, for $i=1,\ldots,c$, we have
  \begin{eqnarray*}
    \phi(f(\x_i))-\phi(f(\x_{i+1}))
      &\geq& \phi'(f(\x_i))(f(\x_i)-f(\x_{i+1})) \\
      &=& C_K|f(\x_i)-f(\x_*)|^{-\theta}(f(\x_i)-f(\x_{i+1})) \\
  \text{[KL property]} &\geq& \|\g(\x_i)\|^{-1}(f(\x_i)-f(\x_{i+1})) \\
  \text{[for \eqref{ccc-2}]} &\geq& \eta\alpha_iC_L\|\g(\x_i)\| \\
      &\geq& \frac{\eta C_L}{2C_U}\|\x_{i+1}-\x_i\|,
  \end{eqnarray*}
  where the last inequality is valid because
  \begin{equation*}
    \|\x_{c+1}-\x_c\|
      \leq 2\left(\|\alpha_c\p_c\|^2-(\alpha_c\x_c^\T\p_c)^2\right)^{\frac{1}{2}}
      \leq 2\alpha_c\|\p_c\|
      \leq 2\alpha_c C_U\|\g(\x_c)\|
  \end{equation*}
  by \eqref{iterate-length} and \eqref{grad-relat}.
  Then,
  \begin{eqnarray*}
    \|\x_{c+1}-\x_*\| &\leq& \sum_{i=1}^c\|\x_{i+1}-\x_i\|+\|\x_1-\x_*\| \\
      &\leq& \frac{2C_U}{\eta C_L}\sum_{i=1}^c \phi(f(\x_i))-\phi(f(\x_{i+1})) + \|\x_1-\x_*\| \\
      &\leq& \frac{2C_U}{\eta C_L}\phi(f(\x_1)) + \|\x_1-\x_*\| \\
      &<& \rho.
  \end{eqnarray*}
  Hence, we get $\x_{c+1}\in\mathscr{B}(\x_*,\rho)$ and \eqref{Lem:KL1a} holds.
  Moreover,
  \begin{equation*}
    \sum_{c=1}^{\infty}\|\x_{c+1}-\x_c\|
      \leq \frac{2C_U}{\eta C_L} \sum_{c=1}^{\infty}\phi(f(\x_c))-\phi(f(\x_{c+1}))
      \leq \frac{2C_U}{\eta C_L}\phi(f(\x_1)).
  \end{equation*}
  The inequality \eqref{Lem:KL1b} is valid.
\end{proof}

\begin{Theorem}\label{Th:KL-converg}
  Suppose that CEST generates an infinite sequence of iterates $\{\x_c\}$.
  Then,
  \begin{equation*}
    \sum_{c=1}^{\infty} \|\x_{c+1}-\x_c\| < +\infty.
  \end{equation*}
  Hence, the total sequence $\{\x_c\}$ has a finite length and
  converges to a unique stationary point.
\end{Theorem}
\begin{proof}
  Owing to the compactness of $\SPHERE$, there exists an accumulate point $\x_*$
  of iterates $\{\x_c\}$. By Theorem \ref{Th:basicConv}, $\x_*$ is also a stationary point.
  Then, there exists an iteration $K$ such that $\rho(\x_K)<\rho$.
  Hence, by Lemma \ref{Lem:KL-pre}, we have
  $\sum_{c=K}^\infty \|\x_{c+1}-\x_c\| < \infty$.
  The proof is complete.
\end{proof}

Next, we estimate the convergence rate of CEST.
The following lemma is useful.

\begin{Lemma}\label{Lem:KL-bd}
  There exists a positive constant $C_m$ such that
  \begin{equation}\label{eq4-13}
    \|\x_{c+1}-\x_c\| \geq C_m\|\g(\x_c)\|.
  \end{equation}
\end{Lemma}
\begin{proof}
  Let $\langle\vt{a},\vt{b}\rangle$ be the angle
  between nonzero vectors $\vt{a}$ and $\vt{b}$, i.e.,
  \begin{equation*}
    \langle\vt{a},\vt{b}\rangle \equiv \arccos\frac{\vt{a}^\T\vt{b}}{\|\vt{a}\|\|\vt{b}\|}
      \in[0,\pi].
  \end{equation*}
  In fact, $\langle \cdot,\cdot \rangle$ is a metric in a unit sphere and
  satisfies the triangle inequality
  \begin{equation*}
    \langle\vt{a},\vt{b}\rangle \leq
      \langle\vt{a},\vt{c}\rangle + \langle\vt{c},\vt{b}\rangle
  \end{equation*}
  for all nonzero vectors $\vt{a},\vt{b}$, and $\vt{c}$.

  From the triangle inequality, we get
  \begin{equation*}
    \langle\x_c,-\g(\x_c)\rangle - \langle-\g(\x_c),\p_c\rangle
    \leq \langle\x_c,\p_c\rangle
    \leq \langle\x_c,-\g(\x_c)\rangle + \langle-\g(\x_c),\p_c\rangle.
  \end{equation*}
  Owing to \eqref{xkgk}, we know $\langle\x_c,-\g(\x_c)\rangle = \frac{\pi}{2}$.
  It yields that
  \begin{equation*}
    \frac{\pi}{2} - \langle-\g(\x_c),\p_c\rangle
    \leq \langle\x_c,\p_c\rangle
    \leq \frac{\pi}{2} + \langle-\g(\x_c),\p_c\rangle.
  \end{equation*}
  Hence, we have
  \begin{equation*}
    \sin \langle\x_c,\p_c\rangle
      \geq \sin\left(\frac{\pi}{2} - \langle-\g(\x_c),\p_c\rangle\right) \\
      = \cos\langle-\g(\x_c),\p_c\rangle \\
      = \frac{-\p_c^\T\g(\x_c)}{\|\p_c\|\|\g(\x_c)\|} \\
      \geq \frac{C_L}{C_U},
  \end{equation*}
  where the last inequality holds because of \eqref{grad-relat}.
  Recalling \eqref{iterate-length} and $\x_c\in\SPHERE$, we obtain
  \begin{equation*}
    \|\x_{c+1}-\x_c\| =
      2\left(\frac{\|\alpha_c\p_c\|^2(1-\cos^2\langle\x_c,\alpha_c\p_c\rangle)}{
            1+\|\alpha_c\p_c\|^2(1-\cos^2\langle\x_c,\alpha_c\p_c\rangle)}\right)^{\frac{1}{2}}
      = \frac{2\alpha_c\|\p_c\|\sin\langle\x_c,\alpha_c\p_c\rangle}{
            \sqrt{1+\alpha_c^2\|\p_c\|^2\sin^2\langle\x_c,\alpha_c\p_c\rangle}}.
  \end{equation*}
  Since $\alpha_{\min}\leq\alpha_c\leq 1$ and $\|\p_c\|\leq C_U\|\g(\x_c)\|\leq C_UM$,
  it yields that
  \begin{equation*}
    \|\x_{c+1}-\x_c\| \geq \frac{2\alpha_{\min}C_LC_U^{-1}}{
      \sqrt{1+C_U^2M^2}}\|\p_c\|
      \geq \frac{2\alpha_{\min}C_L}{C_U(1+C_UM)}\|\p_c\|.
  \end{equation*}
  From \eqref{grad-relat}, we have
  $\|\p_c\|\|\g(\x_c)\|\geq-\p_c^\T\g(\x_c)\geq C_L\|\g(\x_c)\|^2$.
  Hence, $\|\p_c\|\geq C_L\|\g(\x_c)\|$.
  Therefore, this lemma is valid by taking $C_m\equiv \frac{2\alpha_{\min}C_L^2}{C_U(1+C_UM)}$.
\end{proof}

\begin{Theorem}\label{Th:KL-rate}
  Suppose that $\x_*$ is the stationary point of an infinite sequence of iterates $\{\x_c\}$
  generated by CEST. Then, we have the following estimations.
  \begin{itemize}
    \item If $\theta\in(0,\tfrac{1}{2}]$, there exists a $\gamma>0$ and $\varrho\in(0,1)$ such that
      \begin{equation*}
        \|\x_c-\x_*\| \leq \gamma\varrho^c.
      \end{equation*}
    \item If $\theta\in(\tfrac{1}{2},1)$, there exists a $\gamma>0$ such that
      \begin{equation*}
        \|\x_c-\x_*\| \leq \gamma c^{-\frac{1-\theta}{2\theta-1}}.
      \end{equation*}
  \end{itemize}
\end{Theorem}
\begin{proof}
  Because of the validation of Lemma \ref{Lem:KL-bd}, the proof of this theorem
  is similar to \cite[Theorem 2]{AtB-09} and \cite[Theorem 7]{CQW-15}.
\end{proof}

Liu and Nocedal \cite[Theorem 7.1]{LiN-89} proved that L-BFGS converges linearly
if the level set of $f(\x)$ is convex and
the second-order sufficient condition at $\x_*$ holds.
We remark here that, if the second-order sufficient condition holds,
the exponent $\theta=\frac{1}{2}$ in KL property \eqref{KL-ineq}.
According to Theorem \ref{Th:KL-rate}, the infinite sequence of iterates $\{\x_c\}$
has a linear convergence rate.
Hence, to obtain the same local linear convergence rate in theory,
we assume $\theta=\frac{1}{2}$ in KL property
is weaker than the second-order sufficient condition.

\subsection{On the extreme eigenvalue}\label{GloS}

For the target of computing the smallest eigenvalue of
a large scale sparse tensor arising from a uniform hypergraph,
we start CEST from plenty of randomly initial points.
Then, we regard the resulting smallest merit function value as the smallest eigenvalue
of this tensor. The following theorem reveals the successful probability of this strategy.

\begin{Theorem}
  Suppose that we start CEST from $N$ initial points which are sampled from $\SPHERE$ uniformly
  and regard the resulting smallest merit function value as the smallest eigenvalue.
  Then, there exists a constant $\varsigma\in(0,1]$ such that
  the probability of obtaining the smallest eigenvalue is at least
  \begin{equation}\label{Prob}
    1-(1-\varsigma)^N.
  \end{equation}
  Therefore, if the number of samples $N$ is large enough,
  we obtain the smallest eigenvalue with a high probability.
\end{Theorem}
\begin{proof}
  Suppose that $\x^*$ is an eigenvector corresponding to the smallest eigenvalue
  and $\mathscr{B}(\x^*,\rho)\subseteq \mathscr{U}$ is a neighborhood
  as defined in Lemma \ref{Lem:KL-pre}.
  Since the function $\rho(\cdot)$ in \eqref{func-rho} is continuous and
  satisfies $\rho(\x^*)=0<\rho$, there exists a neighborhood
  $\mathscr{V}(\x^*) \equiv \{\x\in\SPHERE :\rho(\x)<\rho\}\subseteq \mathscr{U}$.
  That is to say, if an initial point $\x_1$ happens to be sampled from $\mathscr{V}(\x^*)$,
  the total sequence of iterates $\{\x_c\}$ converges to $\x^*$ by
  Lemma \ref{Lem:KL-pre} and Theorem \ref{Th:KL-converg}.
  Next, we consider the probability of this random event.

  Let $S$ and $A$ be hypervolumes of $(n-1)$ dimensional solids $\SPHERE$ and $\mathscr{V}(\x^*)$
  respectively.\footnote{The hypervolume of the $(n-1)$ dimensional unit sphere is
  $S = \frac{2\pi^{n/2}}{\Gamma(n/2)},$ where $\Gamma(\cdot)$ is
  the Gamma function.} (That is to say, the ``area'' of the surface of $\SPHERE$ in $\REAL^n$
  is $S$ and the ``area'' of the surface of
  $\mathscr{V}(\x^*)\subseteq\SPHERE$ in $\REAL^n$ is $A$. Hence, $A\leq S$.)
  Then, $S$ and $A$ are positive.
  By the geometric probability model,
  the probability of one randomly initial point $\x_1\in\mathscr{V}(\x^*)$ is
  \begin{equation*}
    \varsigma \equiv \frac{A}{S} >0.
  \end{equation*}
  In fact, once $\{\x_c\}\cap\mathscr{V}(\x^*)\neq\emptyset$,
  we could obtain the smallest eigenvalue.
  When starting from $N$ initial points generated by a uniform sample on $\SPHERE$,
  we obtain the probability as (\ref{Prob}).
\end{proof}

If we want to calculate the largest eigenvalue of a tensor $\Ten$,
we only need to replace the merit function $f(\x)$ in \eqref{Sph-Opt} with
\begin{equation*}
    \widehat{f}(\x) = -\frac{\Ten\x^k}{\Btn\x^k}.
\end{equation*}
All of the theorems for the largest eigenvalue of a tensor
could be proved in a similar way.

\section{Numerical experiments}

The novel CEST algorithm is implemented in Matlab and uses the following parameters
\begin{equation*}
    L = 5, \quad \eta = 0.01, \quad\text{and}\quad\beta=0.5.
\end{equation*}
We terminate CEST if
\begin{equation}\label{Tol-1}
    \|\g(\x_c)\|_{\infty} < 10^{-6}
\end{equation}
or
\begin{equation}\label{Tol-2}
    \|\x_{c+1}-\x_c\|_{\infty} < 10^{-8}
    \quad\text{ and }\quad
    \frac{|f(\x_{c+1})-f(\x_c)|}{1+|f(\x_c)|} < 10^{-16}.
\end{equation}
If the number of iterations reaches $5000$, we also stop.
All of the codes are written in Matlab 2012a and
run in a ThinkPad T450 laptop with Intel i7-5500U CPU and 8GB RAM.

We compare the following four algorithms in this section.
\begin{itemize}
  \item An adaptive shifted power method \cite{KoM-11,KoM-14} (Power M.).
    In Tensor Toolbox 2.6,\footnote{See \Path{http://www.sandia.gov/~tgkolda/TensorToolbox/index-2.6.html}.}
    it was implemented as \texttt{eig\_sshopm} and \texttt{eig\_geap}
    for Z- and H-eigenvalues of symmetric tensors respectively.
  \item Han's unconstrained optimization approach (Han's UOA) \cite{Han-13}.
    We solve the optimization model by \texttt{fminunc} in Matlab with settings:
    \texttt{GradObj:on, LargeScale:off, TolX:1.e-8, TolFun:1.e-16, MaxIter:5000, Display:off}.
    Since iterates generated by Han's UOA are not restricted on the unit sphere $\SPHERE$,
    the tolerance parameters are different from other algorithms.
  \item CESTde: we implement CEST for a dense symmetric tensor, i.e.,
    the skills addressed in Section \ref{CompIssue} are not applied.
  \item CEST: the novel method is proposed and analyzed in this paper.
\end{itemize}

For tensors arising from an even-uniform hypergraph, each algorithm starts from
one hundred random initial points sampled from a unit sphere $\SPHERE$ uniformly.
Then, we obtain one hundred estimated eigenvalues $\lambda_1,\ldots,\lambda_{100}$.
If the extreme eigenvalue $\lambda^*$ of that tensor is available,
we count the accuracy rate of this algorithm as
\begin{equation}\label{Truth}
    \mathrm{Accu.} \equiv \left|\left\{i \,:\, \frac{|\lambda_i-\lambda^*|}{1+|\lambda^*|}
      \leq 10^{-8}\right\}\right| \times 1\%.
\end{equation}
After using the global strategy in Section \ref{GloS},
we regard the best one as the estimated extreme eigenvalue.

\subsection{Small hypergraphs}

First, we investigate some extreme eigenvalues of symmetric tensors
corresponding to small uniform hypergraphs.

\smallskip\indent\textbf{Squid.}
A squid $G_S^k=(V,E)$ is a $k$-uniform hypergraph which has $(k^2-k+1)$ vertices
and $k$ edges: legs $\{i_{1,1},\ldots,i_{1,k}\}, \ldots, \{i_{k-1,1},\ldots,i_{k-1,k}\}$
and a head $\{i_{1,1},\ldots,i_{k-1,1},i_k\}$.
When $k$ is even, $G_S^k$ is obviously connected and odd-bipartite.
Hence, we have $\lambda_{\min}^H(\Adj(G_S^k)) = -\lambda_{\max}^H(\Adj(G_S^k))$
because of Theorem \ref{Th:Spectrum}(iv).
Since the adjacency tensor $\Adj(G_S^k)$ is nonnegative and weakly irreducible,
its largest H-eigenvalue $\lambda_{\max}^H(\Adj(G_S^k))$ could be computed by
the Ng-Qi-Zhou algorithm \cite{NQZ-09}.
For the smallest H-eigenvalue of $\Adj(G_S^k)$,
we perform the following tests.

With regards to the parameter $L$ for L-BFGS, Nocedal suggested that
L-BFGS performs well when $3\leq L \leq7$.
Hence, we compare L-BFGS with $L=3,5,7$ and
the Barzilai-Borwein method ($L=0$).
The parameter $\gamma_c$ is chosen from
$\gamma_c^{BB1}$, $\gamma_c^{BB2}$, and $\gamma_c^{Dai}$ randomly.
For $k$-uniform squids with $k=4,6,8$, we compute the smallest
H-eigenvalues of their adjacency tensors.
The total CPU times for one hundred runs are illustrated in Figure \ref{Squid-D}.
Obviously, L-BFGS is about five times faster than the Barzilai-Borwein method.
Following Nocedal's setting\footnote{See \Path{http://users.iems.northwestern.edu/~nocedal/lbfgs.html}.},
we prefer to set $L=5$ in CEST.

\begin{figure}[!tb]
  \centering
  \includegraphics[width=.7\textwidth]{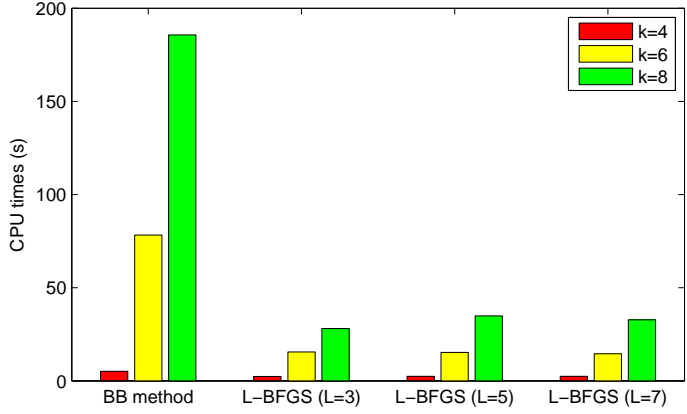}
  \caption{CPU times for L-BFGS with different parameters $L$.}\label{Squid-D}
\end{figure}

\begin{figure}[!tb]
  \centering
  \includegraphics[width=.4\textwidth]{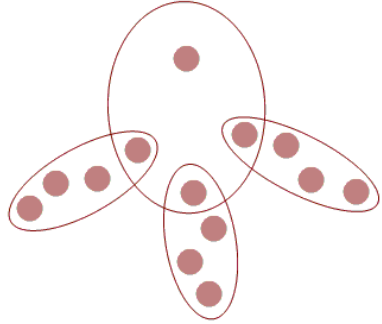}
  \caption{A $4$-uniform squid $G_S^4$.}\label{Squid-A}
\end{figure}

\begin{table}[!tb]
  \caption{Comparison results for the smallest H-eigenvalue of
    the adjacency tensor $\Adj(G_S^4)$.}\label{Squid-B}
\begin{center}
\begin{tabular}{l|c|c|c}
  \hline
  Algorithms  & $\lambda^H_{\min}(\Adj(G_S^4))$ & Time (s) & Accu. \\
  \hline
  Power M.  & $-1.3320$ & 97.20 & 100\% \\
  Han's UOA & $-1.3320$ & 21.20 & 100\% \\
  CESTde    & $-1.3320$ & 35.72 & 100\% \\
  CEST      & $-1.3320$ & \hspace{1ex}2.43 & 100\% \\
  \hline
\end{tabular}
\end{center}
\end{table}

Next, we consider the $4$-uniform squid $G_S^4$ illustrated in Figure \ref{Squid-A}.
For reference, we remind $\lambda_{\max}^H(\Adj(G_S^k))=1.3320$
by the Ng-Qi-Zhou algorithm. Then, we compare
four kinds of algorithms: Power M., Han's UOA, CESTde, and CEST.
Results are shown in Table \ref{Squid-B}.
Obviously, all algorithms find the smallest H-eigenvalue of
the adjacency tensor $\Adj(G_S^4)$ with probability $1$.
Compared with Power M., Han's UOA and CESTde save $78\%$ and $63\%$
CPU times, respectively.
When the sparse structure of the adjacency tensor $\Adj(G_S^4)$ is explored,
CEST is forty times faster than the power method.

%

\smallskip\indent
\textbf{Blowing up the Petersen graph.}
Figure \ref{FpStar-A} illustrates an ordinary graph $G_P$: the Petersen graph.
It is non-bipartite and the smallest eigenvalue of its signless Laplacian matrix is one.
We consider the $2k$-uniform hypergraph $G_P^{2k,k}$ which is generated by blowing up
each vertex of $G_P$ to a $k$-set. Hence, $G_P^{2k,k}$ contains $10k$ vertices
and $15$ edges.
From Theorem \ref{PowerG}, we know that the smallest H-eigenvalue of
the signless Laplacian tensor $\sLp(G_P^{2k,k})$ is exactly one.

\begin{figure}[!tb]
  \centering
  \includegraphics[width=.4\textwidth]{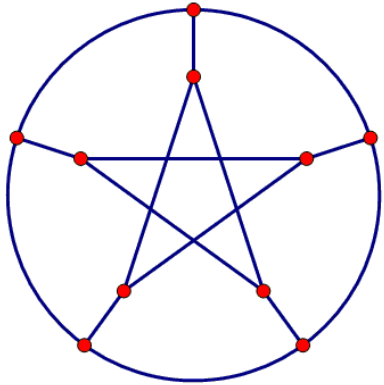}\\
  \caption{The Petersen graph $G_P$.}\label{FpStar-A}
\end{figure}

\begin{table}[!tb]
  \caption{Comparison results for the smallest H-eigenvalue of
    the signless Laplacian tensor $\sLp(G_P^{4,2})$.
    (*) means a failure.}\label{FpStar-B}
\begin{center}
\begin{tabular}{l|c|c|c}
  \hline
  Algorithms  & $\lambda^H_{\min}(\sLp(G_P^{4,2}))$ & CPU time(s) & Accu. \\
  \hline
  Power M.  & 1.0000 & 657.44 &  95\% \\
  Han's UOA & 1.1877$^{(*)}$ &  \hspace{1ex}93.09 &       \\
  CESTde    & 1.0000 & \hspace{1ex}70.43 & 100\% \\
  CEST      & 1.0000 & \hspace{2ex}3.82 & 100\% \\
  \hline
\end{tabular}
\end{center}
\end{table}

Table \ref{FpStar-B} reports comparison results of four sorts of algorithms
for the $4$-uniform hypergraph $G_P^{4,2}$.
Here, Han's UOA missed the smallest H-eigenvalue of the signless Laplacian tensor.
Power M., CESTde, and CEST find the true solution with a high probability.
When compared with power M., CESTde saves more than $88\%$ CPU times.
Moreover, the approach exploiting the sparsity
improves CESTde greatly, since CEST saves about $99\%$ CPU times.

\begin{table}[!tb]
  \caption{CEST computes the smallest H-eigenvalues of
    signless Laplacian tensors $\sLp(G_P^{2k,k})$.}\label{FpStar-C}
\begin{center}
\begin{tabular}{c|cccccccccc}
  \hline
  $2k$           &  2  &  4  &  6  &  8  &  10 &  12 &  14 &  16 &  18 &  20 \\
  \hline
  Accu. (\%) & 100 & 100 & 100 & 100 &  99 &  98 &  86 &  57 &  20 &   4 \\
  \hline
\end{tabular}
\end{center}
\end{table}

For $2k$-uniform hypergraph $G_P^{2k,k}$ with $k=1,\ldots,10$,
we apply CEST for computing the smallest H-eigenvalues of
their signless Laplacian tensors. Detailed results are shown
in Table \ref{FpStar-C}. For each case, CEST finds the
smallest H-eigenvalue of the signless Laplacian tensor in at most one minute.
With the increment of $k$, the percentage of accurate estimations
decreases.

\smallskip\indent
\textbf{Grid hypergraphs.}
The grid $G_G^s$ is a $4$-uniform hypergraph generated by subdividing a square.
If the subdividing order $s=0$, the grid $G_G^0$ is the square
with $4$ vertices and only one edge.
When the subdividing order $s\geq1$, we subdivide each edge of $G_G^{s-1}$ into four edges.
Hence, a grid $G_G^s$ has $(2^s+1)^2$ vertices and $4^s$ edges.
The $4$-uniform grid $G_G^s$ with $s=1,2,3,4$ are illustrated in Figure \ref{Matts-A}.

\begin{figure}[!tb]
\begin{center}
\begin{tabular}{cccc}
  \includegraphics[width=.12\textwidth]{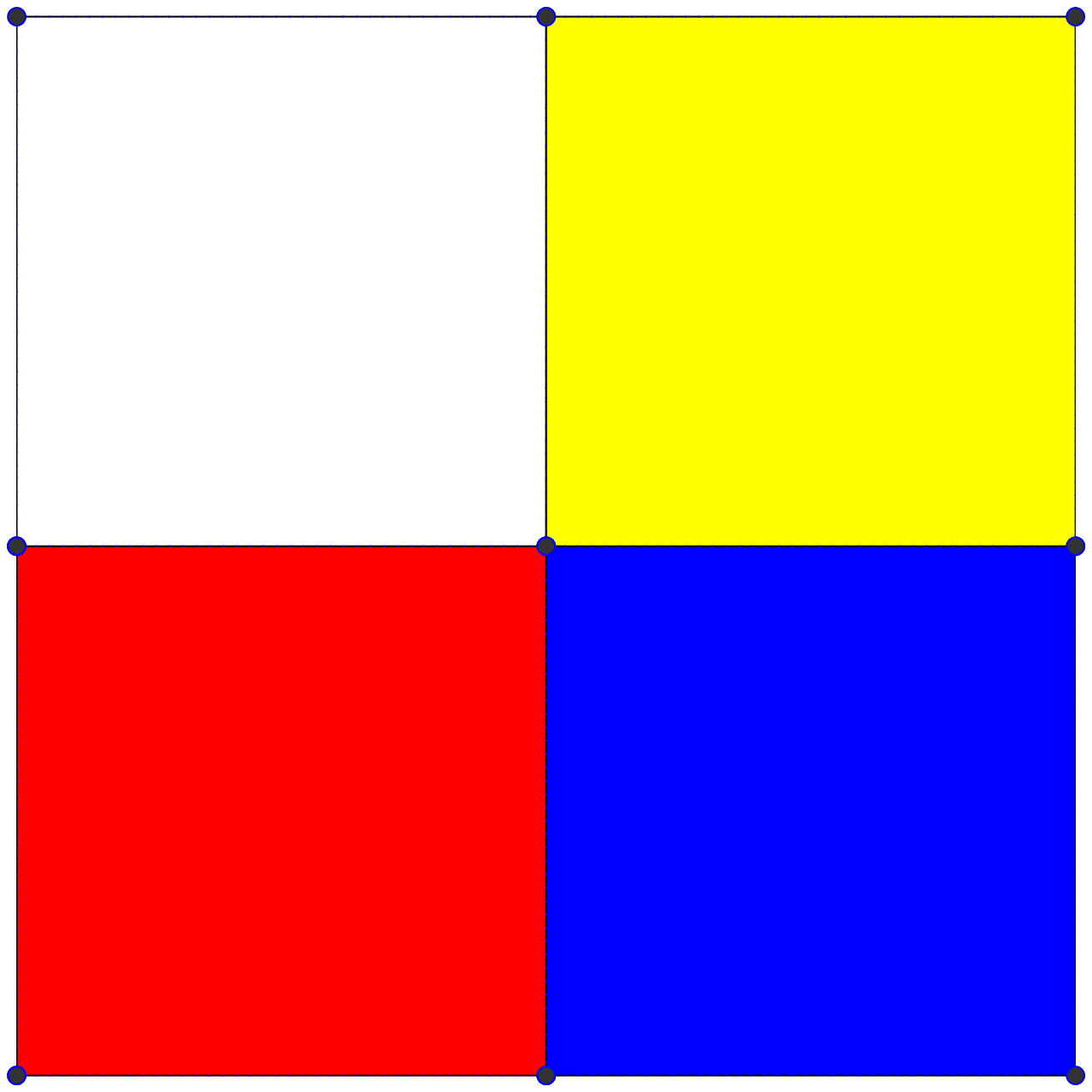} &
    \includegraphics[width=.19\textwidth]{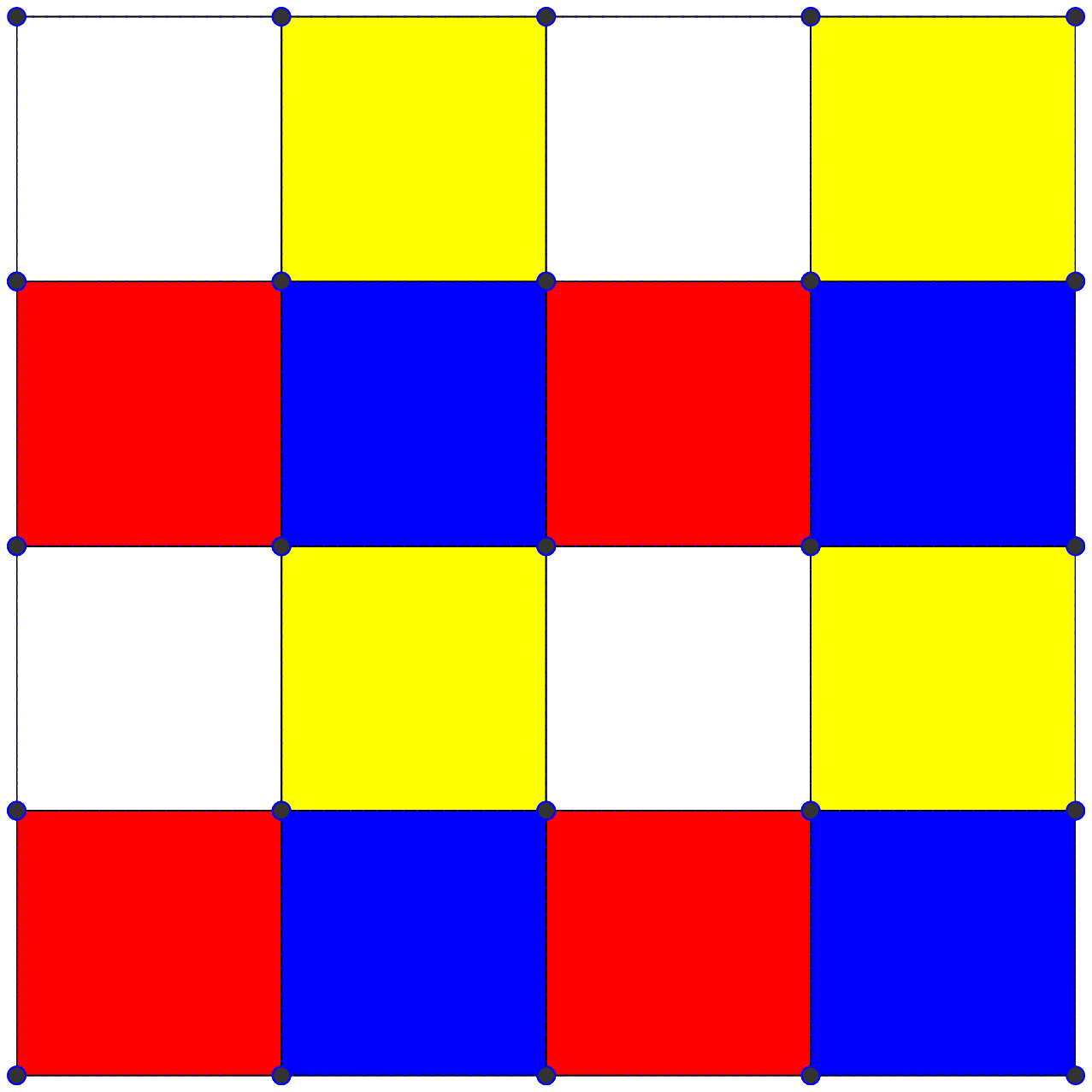} &
      \includegraphics[width=.26\textwidth]{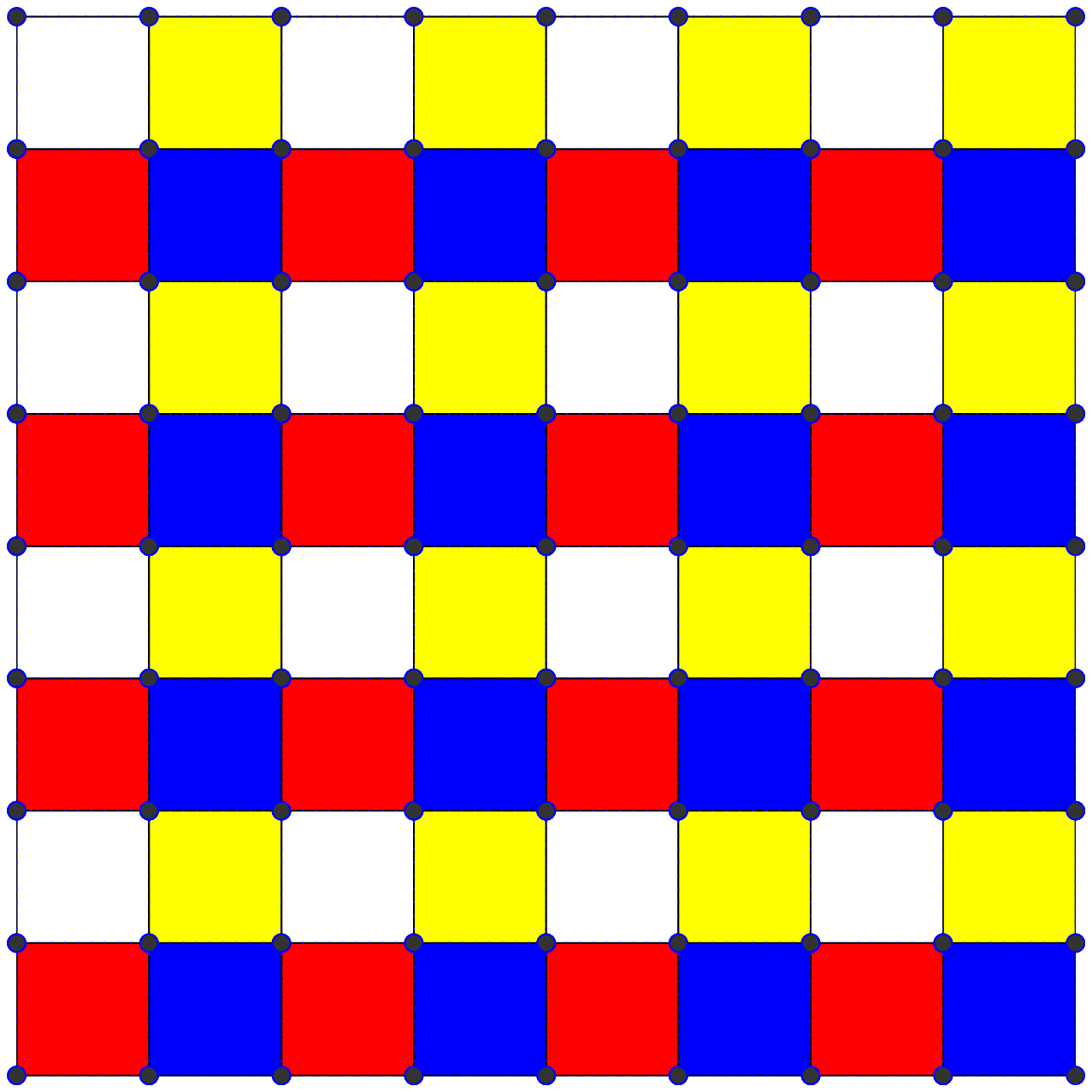} &
        \includegraphics[width=.33\textwidth]{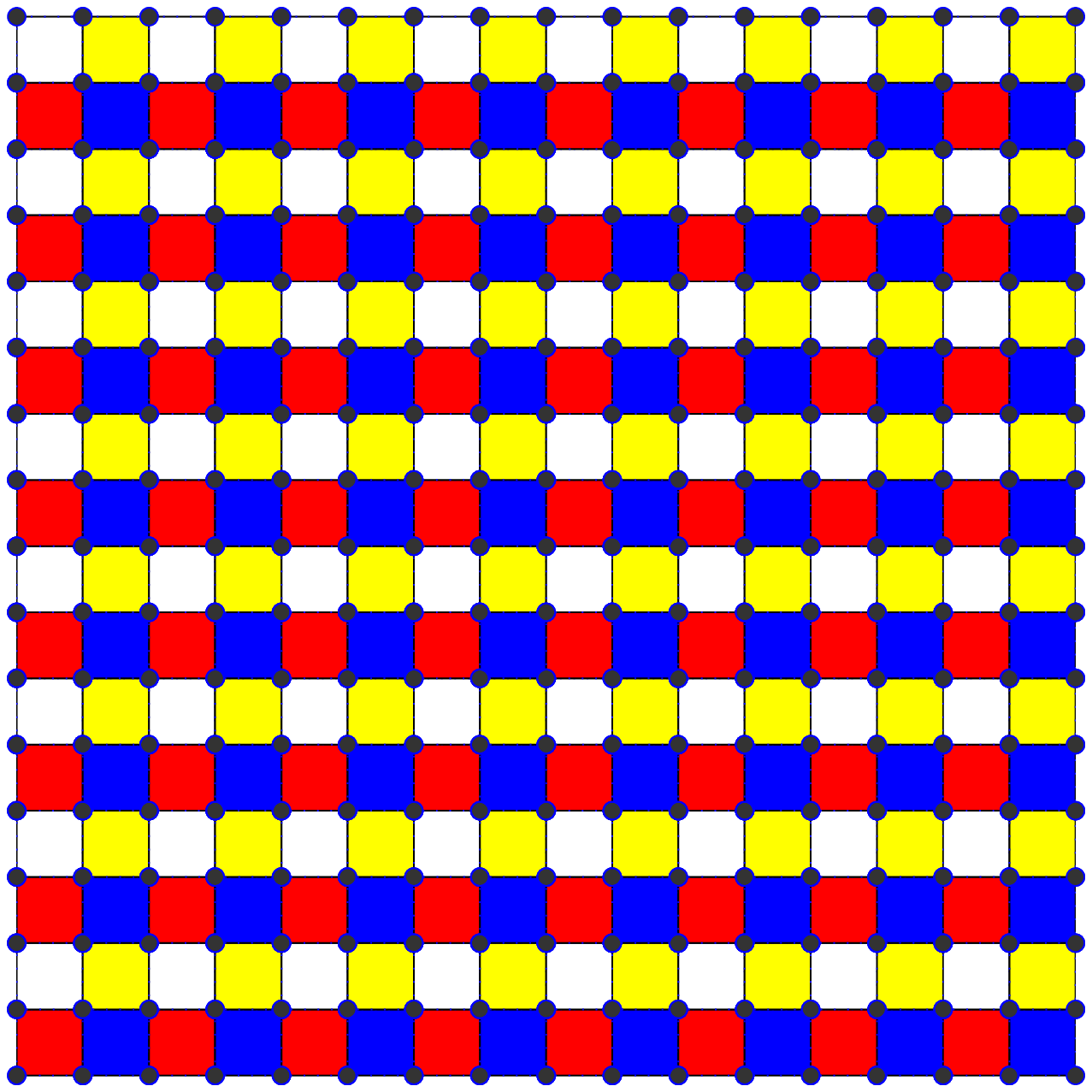} \\
  (a) $s=1$ & (b) $s=2$ &
  (c) $s=3$ & (d) $s=4$ \\
\end{tabular}
\end{center}
  \caption{Some $4$-uniform grid hypergraphs.}\label{Matts-A}
\end{figure}

\begin{table}[!tb]
  \caption{Comparison results for the largest H-eigenvalue of
    the Laplacian tensor $\Lap(G_G^2)$.}\label{Matts-B}
\begin{center}
\begin{tabular}{l|c|r|c}
  \hline
  Algorithms  & $\lambda^H_{\max}(\Lap(G_G^2))$ & Time(s) & Accu. \\
  \hline
  Power M.  & $6.5754$ & 142.51 & 100\% \\
  Han's UOA & $6.5754$ &  35.07 & 100\% \\
  CESTde   & $6.5754$ &  43.35 & 100\% \\
  CEST    & $6.5754$ &   2.43 & 100\% \\
  \hline
\end{tabular}
\end{center}
\end{table}

\begin{table}[!tb]
  \caption{CEST computes the largest H-eigenvalue of
    Laplacian tensors $\Lap(G_G^s)$.}\label{Matts-C}
\begin{center}
\begin{tabular}{crr|c|rr|r}
  \hline
   $s$ &  $n$ &  $m$ & $\lambda^H_{\max}(\Lap(G_G^s))$ & Iter. & Time(s) & Accu. \\
  \hline
    1  &    9 &    4 & $4.6344$ &   2444 &   1.39 & 100\% \\
    2  &   25 &   16 & $6.5754$ &   4738 &   2.43 & 100\% \\
    3  &   81 &   64 & $7.5293$ &  12624 &   6.44 &  98\% \\
    4  &  289 &  256 & $7.8648$ &  34558 &  26.08 &  65\% \\
  \hline
\end{tabular}
\end{center}
\end{table}

We study the largest H-eigenvalue of the Laplacian tensor of a $4$-uniform grid $G_G^2$
as shown in Figure \ref{Matts-A}(b).
Obviously, the grid $G_G^2$ is connected and odd-bipartite.
Hence, we have $\lambda_{\max}^H(\Lap(G_G^2))  = \lambda_{\max}^H(\sLp(G_G^2))$
by Theorem \ref{Th:Spectrum}(ii).
Using the Ng-Qi-Zhou algorithm, we calculate $\lambda_{\max}^H(\sLp(G_G^2))=6.5754$ for reference.
Table \ref{Matts-B} shows the performance of four kinds of algorithms:
Power M., Han's UOA, CESTde, and CEST.
All of them find the largest H-eigenvalue of $\Lap(G_G^2)$ with probability one.
Compared with Power M., Han's UOA and CESTde saves about
$75\%$ and $70\%$ CPU times respectively.
CEST is about fifty times faster than the Power Method.

In Table \ref{Matts-C}, we show the performance of CEST for
computing the largest H-eigenvalues of
the Laplacian tensors of grids illustrated in Figure \ref{Matts-A}.

\subsection{Large hypergraphs}

Finally, we consider two large scale even-uniform hypergraphs.

\smallskip\indent\textbf{Sunflower.}
A $k$-uniform sunflower $G_{S}=(V,E)$ with a maximum degree $\Delta$ has
$n=(k-1)\Delta+1$ vertices and $\Delta$ edges, where
$V=V_0 \cup V_1 \cup \cdots \cup V_{\Delta}$, $|V_0|=1$,
$|V_i|=k-1$ for $i=1,\ldots,\Delta$, and
$E=\{V_0 \cup V_i \,|\, i=1,\ldots,\Delta\}$.
Figure \ref{HyperGraph-1} is a $4$-uniform sunflower with $\Delta=3$.

Hu et al. \cite{HQX-15} argued in the following theorem that
the largest H-eigenvalue of the Laplacian tensor of
an even-uniform sunflower has a closed form solution.

\begin{Theorem}\label{Th:SunFlower}(Theorems 3.2 and 3.4 of \cite{HQX-15})
  Let $G$ be a $k$-graph with $k\geq4$ being even.
  Then $$\lambda_{\max}^H(\Lap) \geq \lambda_H^*,$$
  where $\lambda_H^*$ is the unique real root of the equation
  $(1-\lambda)^{k-1}(\lambda-\Delta)+\Delta=0$ in the interval $(\Delta,\Delta+1)$.
  The equality holds if and only if $G$ is a sunflower.
\end{Theorem}

We aim to apply CEST for computing the largest H-eigenvalues of
Laplacian tensors of even-uniform sunflowers.
For $k=4$ and $6$, we consider sunflowers with maximum degrees from ten to one million.
Since we deal with large scale tensors, we slightly
enlarge tolerance parameters in \eqref{Tol-1} and \eqref{Tol-2} by multiplying $\sqrt{n}$.
To show the accuracy of the estimated H-eigenvalue $\lambda^H_{\max}(\Lap)$,
we calculate the relative error
$$\mathrm{RE} = \frac{|\lambda^H_{\max}(\Lap)-\lambda_H^*|}{\lambda_H^*},$$
where $\lambda_H^*$ is defined in Theorem \ref{Th:SunFlower}.
Table \ref{SunFl-C} reports detailed numerical results.
Obviously, the largest H-eigenvalues of Laplacian tensors
returned by CEST have a high accuracy.
Relative errors are in the magnitude $\mathcal{O}(10^{-10})$.
The CPU time costed by CEST does not exceed $80$ minutes.

\begin{table}
  \caption{The largest H-eigenvalues of Laplacian tensors
    corresponding to $k$-uniform sunflowers.}\label{SunFl-C}
\begin{center}
\begin{tabular}{cr|r@{.}l c|rr|r}
  \hline
  $k$ &  $n$ & \multicolumn{2}{c}{$\lambda^H_{\max}(\Lap)$} & RE & Iter. & Time(s) & Accu. \\
  \hline
  $4$ &        31 &        10 & 0137 & $5.3218\times10^{-16}$ & 4284 &    2.39 & 100\% \\
      &       301 &       100 & 0001 & $7.3186\times10^{-14}$ & 4413 &    3.73 &  42\% \\
      &     3,001 &     1,000 & 0000 & $1.2917\times10^{-10}$ & 1291 &    4.84 & 100\% \\
      &    30,001 &    10,000 & 0000 & $5.9652\times10^{-12}$ & 1280 &   38.14 & 100\% \\
      &   300,001 &   100,000 & 0000 & $9.6043\times10^{-15}$ & 1254 &  512.04 & 100\% \\
      & 3,000,001 & 1,000,000 & 0000 &                    $0$ & 1054 & 4612.28 & 100\% \\
  \hline
  $6$ &        51 &        10 & 0002 & $2.4831\times10^{-12}$ & 4768 &    3.34 &   8\% \\
      &       501 &       100 & 0000 & $2.4076\times10^{-10}$ & 1109 &    1.47 &  98\% \\
      &     5,001 &     1,000 & 0000 & $3.2185\times10^{-13}$ & 1020 &    5.85 & 100\% \\
      &    50,001 &    10,000 & 0000 & $5.7667\times10^{-12}$ &  927 &   44.62 & 100\% \\
      &   500,001 &   100,000 & 0000 & $1.1583\times10^{-13}$ &  778 &  479.52 & 100\% \\
      & 5,000,001 & 1,000,000 & 0000 & $2.3283\times10^{-16}$ &  709 & 4679.30 & 100\% \\
  \hline
\end{tabular}
\end{center}
\end{table}

\smallskip\indent\textbf{Icosahedron.}
An icosahedron has twelve vertices and twenty faces.
The subdivision of an icosahedron could be used to approximate a unit sphere.
The $s$-order subdivision of an icosahedron has $(20\times 4^s)$ faces
and each face is a triangle. We regard three vertices of the triangle as well as its center
as an edge of a $4$-graph $G_I^s$.
Then, the $4$-graph $G_I^s$ must be connected and odd-bipartite. See Figure \ref{Icos-A}.

\begin{figure}[!tb]
\begin{center}
\begin{tabular}{ccc}
  \raisebox{5pt}{\includegraphics[width=.28\textwidth]{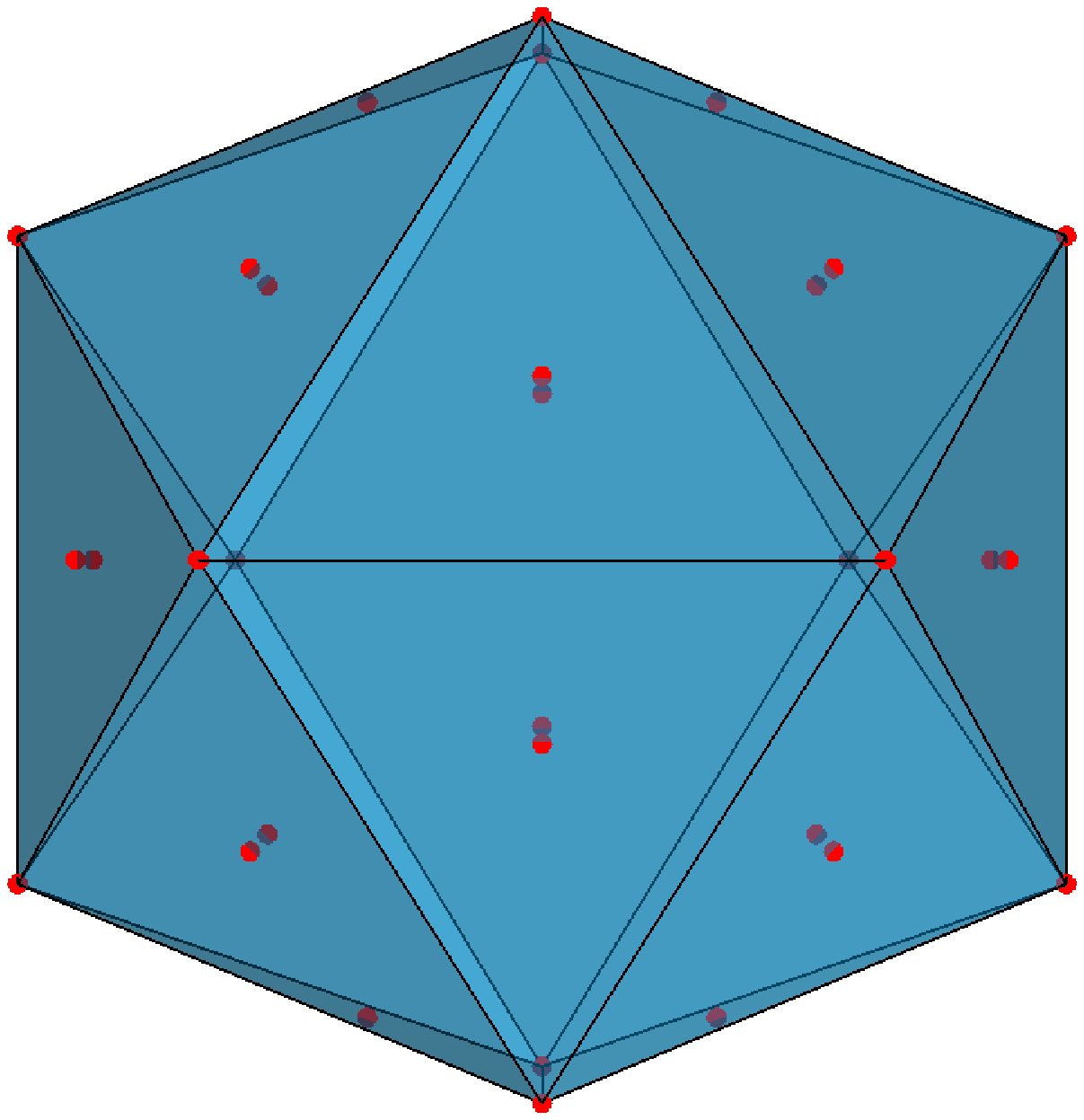}} &
    \includegraphics[width=.3\textwidth]{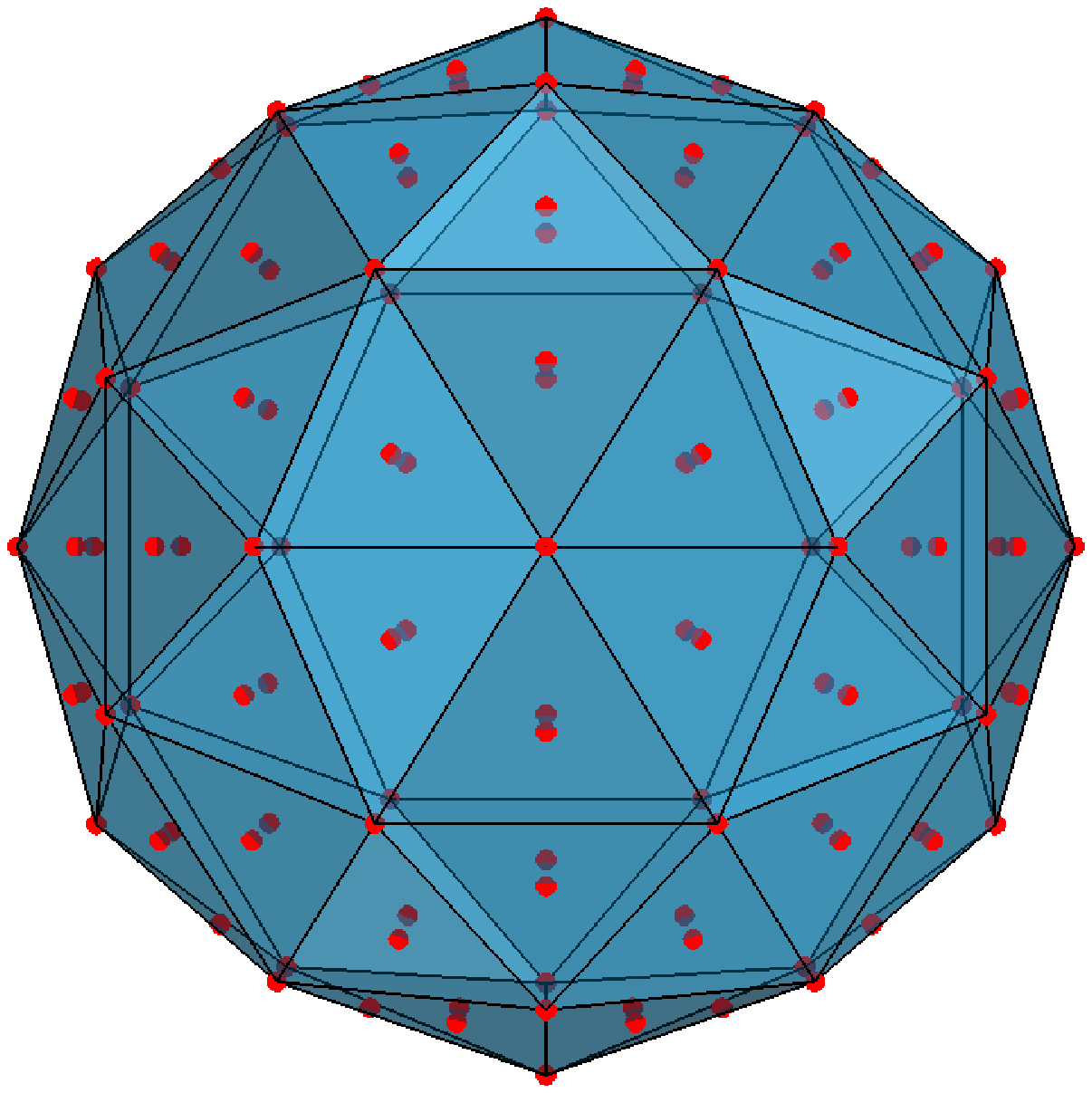} &
      \includegraphics[width=.3\textwidth]{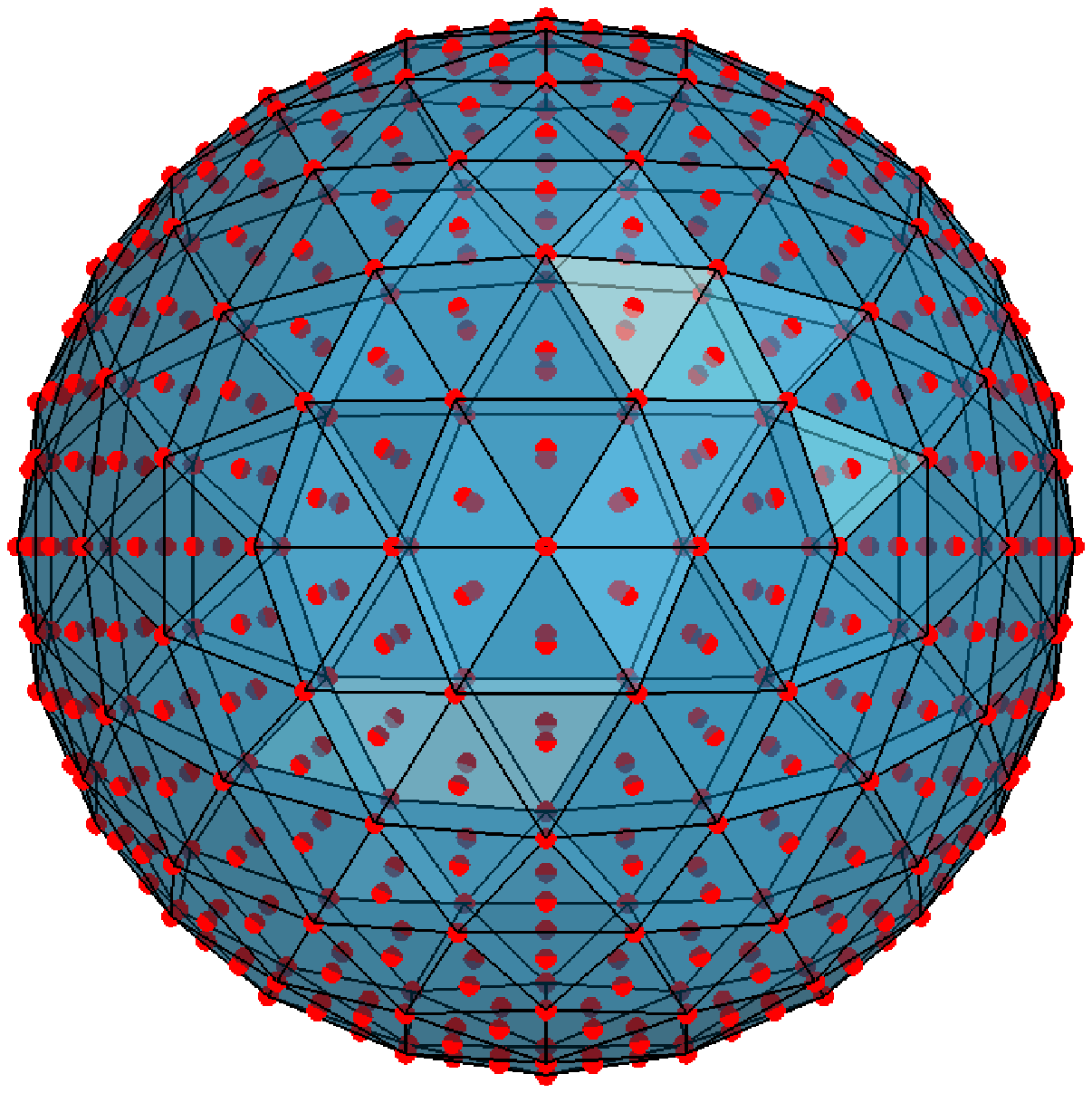} \\
  Icosahedron ($s=0$) & $s=1$ & $s=2$ \\
  & \includegraphics[width=.3\textwidth]{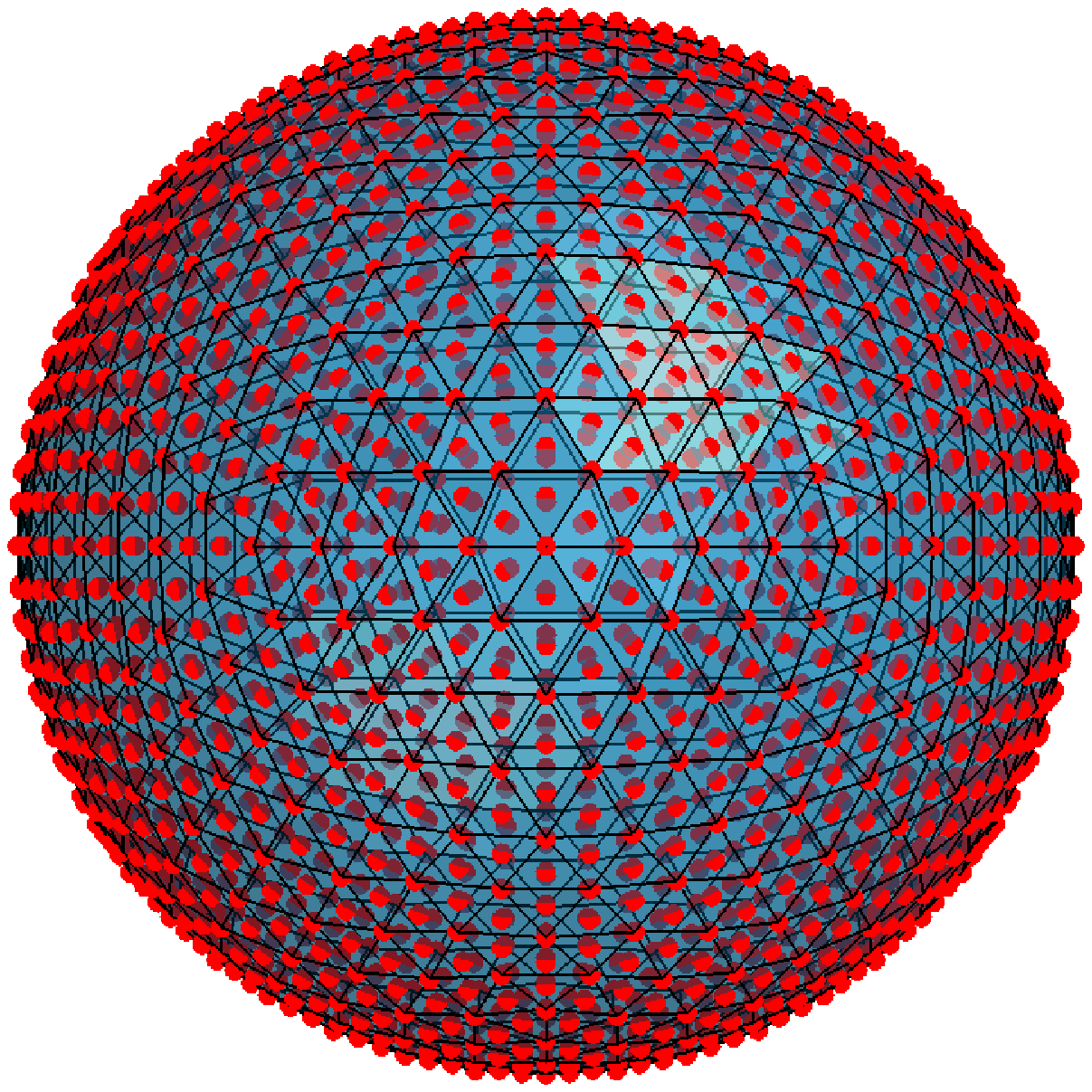} &
    \includegraphics[width=.3\textwidth]{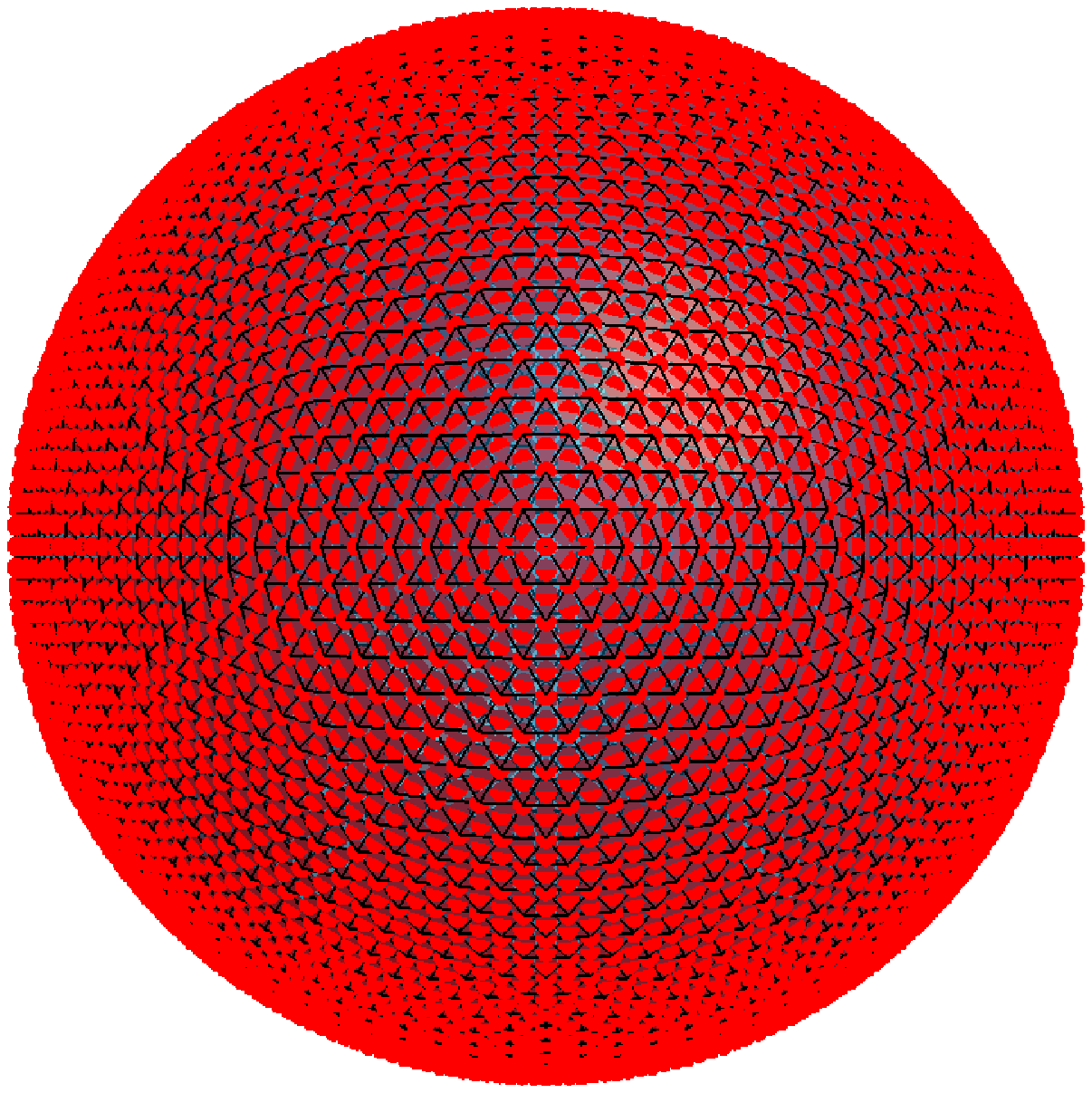} \\
        & $s=3$ & $s=4$ \\
\end{tabular}
\end{center}
  \caption{$4$-uniform hypergraphs: subdivision of an icosahedron.}\label{Icos-A}
\end{figure}

\begin{table}[!tbh]
  \caption{CEST computes the largest Z-eigenvalues of
    Laplacian tensors $\Lap(G_I^s)$ and signless Laplacian tensors $\sLp(G_I^s)$.}\label{Icos-C}
\begin{center}
\begin{tabular}{rrr|ccr|ccr}
  \hline
      $s$ & $n$ & $m$ & $\lambda_{\max}^Z(\Lap(G_I^s))$ & Iter. & time(s) &  $\lambda_{\max}^Z(\sLp(G_I^s))$ & Iter. & time(s)  \\
    \hline
       0 &      32 &        20 & 5 & 1102 &   0.89 &  5 & 1092 &   0.75  \\
       1 &     122 &        80 & 6 & 1090 &   1.09 &  6 & 1050 &   0.75  \\
       2 &     482 &       320 & 6 & 1130 &   1.39 &  6 & 1170 &   1.23  \\
       3 &   1,922 &     1,280 & 6 & 1226 &   3.15 &  6 & 1194 &   2.95  \\
       4 &   7,682 &     5,120 & 6 & 1270 &  10.11 &  6 & 1244 &  10.06  \\
       5 &  30,722 &    20,480 & 6 & 1249 &  36.89 &  6 & 1282 &  35.93  \\
       6 & 122,882 &    81,920 & 6 & 1273 & 166.05 &  6 & 1289 & 161.02  \\
       7 & 491,522 &   327,680 & 6 & 1300 & 744.08 &  6 & 1327 & 739.01  \\
       8 &1,966,082& 1,310,720 & 6 &  574 &1251.36 &  6 &  558 &1225.87  \\
  \hline
\end{tabular}
\end{center}
\end{table}

According to Theorem \ref{Th:Spectrum}(v),
we have $\lambda_{\max}^Z(\Lap(G_I^s)) = \lambda_{\max}^Z(\sLp(G_I^s))$,
although they are unknown.
Experiment results are reported in Table \ref{Icos-C}.
It is easy to see that CEST could compute the largest Z-eigenvalues of both
Laplacian tensors and signless Laplacian tensors of
hypergraphs $G_I^s$ with dimensions up to almost two millions.
In each case of our experiment, CEST costs at most twenty-one minutes.

Additionally, for $4$-graphs $G_I^s$ generated by subdividing an icosahedron,
the following equality seems to hold
\begin{equation}\label{Z-eq}
    \lambda_{\max}^Z(\Lap(G_I^s)) = \lambda_{\max}^Z(\sLp(G_I^s)) = \Delta.
\end{equation}
Bu et al. \cite{BFZ-15} proved that \eqref{Z-eq} holds
for a $k$-uniform sunflower with $3\leq k \leq 2\Delta$.
However, it is an open problem whether the equality \eqref{Z-eq} hold for
a general connected odd-bipartite uniform hypergraph.

%


\section{Conclusion}

Motivated by recent advances in spectral hypergraph theory,
we proposed an efficient first-order optimization algorithm CEST
for computing extreme H- and Z-eigenvalues of sparse tensors
arising form large scale uniform hypergraphs.
Due to the algebraic nature of tensors, we could apply
the Kurdyka-{\L}ojasiewicz property in analyzing the convergence
of the sequence of iterates generated by CEST. By using a simple global strategy,
we prove that the extreme eigenvalue of a symmetric tensor could be touched
with a high probability.

We establish a fast computational framework for products of a vector and
large scale sparse tensors arising from a uniform hypergraph. By using this technique,
the storage of a hypergraph is economic and
the computational cost of CEST in each iteration is cheap.
Numerical experiments show that the novel algorithm CEST could
deal with uniform hypergraphs with millions of vertices.

\appendix
\section*{Appendix}
In this appendix, we will prove that L-BFGS produces a gradient-related direction,
i.e., Theorem \ref{Th: grad-dir} is valid.
First, we consider the classical BFGS update \eqref{def-sy}--\eqref{BFGS-formula}
 and establish the following two lemmas.

\begin{Lemma}\label{Lem:bfgs-ub}
  Suppose that $H_c^+$ is generated by BFGS
  \eqref{def-sy}--\eqref{BFGS-formula}.
  Then, we have
  \begin{equation}\label{ubd-rec}
    \|H_c^+\| \leq \|H_c\|\left(1+\frac{4M}{\kappa_{\epsilon}}\right)^2
      +\frac{4}{\kappa_{\epsilon}}.
  \end{equation}
\end{Lemma}
\begin{proof}
  If $\y_c^\T\s_c<\kappa_{\epsilon}$, we get $\rho_c=0$ and $H_c^+=H_c$.
  Hence, the inequality \eqref{ubd-rec} holds.

  Next, we consider the case $\y_c^\T\s_c\geq\kappa_{\epsilon}$.
  Obviously, $\rho_c\leq\frac{1}{\kappa_{\epsilon}}$.
  From Lemma \ref{Lem:fg-bnd} and all iterates $\x_c\in\SPHERE$, we get
  \begin{equation}\label{skyk}
    \|\s_c\|\leq 2 \quad\text{ and }\quad \|\y_c\|\leq 2M.
  \end{equation}
  Since
  \begin{equation*}
    \|V_c\| \leq 1+\rho_c\|\y_c\|\|\s_c\| \leq 1+\frac{4M}{\kappa_{\epsilon}}
  \quad\text{ and }\quad
    \|\rho_c\s_c\s_c^\T\| \leq \rho_c\|\s_c\|^2 \leq \frac{4}{\kappa_{\epsilon}},
  \end{equation*}
  we have
  \begin{equation*}
    \|H_c^+\| \leq \|H_c\| \|V_c\|^2 + \|\rho_c\s_c\s_c^\T\|
    \leq \|H_c\|\left(1+\frac{4M}{\kappa_{\epsilon}}\right)^2
      +\frac{4}{\kappa_{\epsilon}}.
  \end{equation*}
  Hence, the inequality \eqref{ubd-rec} is valid.
\end{proof}

\begin{Lemma}\label{Lem:bfgs-lb}
  Suppose that $H_c$ is positive definite and $H_c^+$ is generated by BFGS
  \eqref{def-sy}--\eqref{BFGS-formula}.
  Let $\mu_{\min}(H)$ be the smallest eigenvalue of a symmetric matrix $H$.
  Then, we get $H_c^+$ is positive definite and
  \begin{equation}\label{min-eig}
    \mu_{\min}(H_c^+)\geq
      \frac{\kappa_{\epsilon}}{\kappa_{\epsilon}+4M^2\|H_c\|}\mu_{\min}(H_c).
  \end{equation}
\end{Lemma}
\begin{proof}
  For any unit vector $\z$, we have
  \begin{equation*}
    \z^\T H_c^+\z = (\z-\rho_c\s_c^\T\z\y_c)^\T H_c(\z-\rho_c\s_c^\T\z\y_c)+\rho_c(\s_c^\T\z)^2.
  \end{equation*}
  Let $t \equiv \s_c^\T\z$ and
  \begin{equation*}
    \phi(t) \equiv (\z-t\rho_c\y_c)^\T H_c(\z-t\rho_c\y_c)+\rho_ct^2.
  \end{equation*}
  Because $H_c$ is positive definite, $\phi(t)$ is convex and
  attaches its minimum at $t_*=\frac{\rho_c\y_c^\T H_c\z}{\rho_c+\rho_c^2\y_c^\T H_c\y_c}$.
  Hence,
  \begin{eqnarray*}
    \z^\T H_c^+\z &\geq& \phi(t_*) \\
      &=& \z^\T H_c\z-t_*\rho_c\y_c^\T H_c\z \\
      &=& \frac{\rho_c\z^\T H_c\z+\rho_c^2(\y_c^\T H_c\y_c\z^\T H_c\z-(\y_c^\T H_c\z)^2)
            }{\rho_c+\rho_c^2\y_c^\T H_c\y_c} \\
      &\geq& \frac{\z^\T H_c\z}{1+\rho_c\y_c^\T H_c\y_c} >0,
  \end{eqnarray*}
  where the penultimate inequality holds because
  the Cauchy-Schwarz inequality is valid for
  the positive definite matrix norm $\|\cdot\|_{H_c}$, i.e.,
  $\|\y_c\|_{H_c}\|\z\|_{H_c} \geq \y_c^\T H_c\z$.
  Therefore, $H_c^+$ is also positive definite.
  From \eqref{skyk}, it is easy to verify that
  \begin{equation*}
    1+\rho_c\y_c^\T H_c\y_c \leq 1+\frac{4M^2\|H_c\|}{\kappa_{\epsilon}}.
  \end{equation*}
  Therefore, we have
  $\z^\T H_c^+\z\geq\frac{\kappa_{\epsilon}}{\kappa_{\epsilon}+4M^2\|H_c\|}\mu_{\min}(H_c).$
  Hence, we get the validation of \eqref{min-eig}.
\end{proof}

Second, we turn to L-BFGS.
Regardless of the selection of $\gamma_c$ in \eqref{Hk-start},
we get the following lemma.


\begin{Lemma}\label{Lem:gamma-bd}
  Suppose that the parameter $\gamma_c$ takes Barzilai-Borwein steps \eqref{BB-steps}
  or its geometric mean \eqref{Dai-steps}. Then, we have
  \begin{equation}\label{gamma-bd}
    \frac{\kappa_{\epsilon}}{4M^2} \leq \gamma_c \leq \frac{4}{\kappa_{\epsilon}}.
  \end{equation}
\end{Lemma}
\begin{proof}
  If $\y_c^\T\s_c < \kappa_{\epsilon}$, we get $\gamma_c=1$ which satisfies
  the bounds in \eqref{gamma-bd} obviously.

  Otherwise, we have $\kappa_{\epsilon} \leq \y_c^\T\s_c \leq  \|\y_c\|\|\s_c\|$.
  Recalling \eqref{skyk}, we get
  \begin{equation*}
    \frac{\kappa_{\epsilon}}{2}\leq\|\y_c\|\leq 2M \qquad\text{ and }\qquad
    \frac{\kappa_{\epsilon}}{2M}\leq\|\s_c\|\leq 2.
  \end{equation*}
  Hence, we have
  \begin{equation*}
    \frac{\kappa_{\epsilon}}{4M^2} \leq \frac{\y_c^\T\s_c}{\|\y_c\|^2}
      \leq \frac{\|\s_c\|\|\y_c\|}{\|\y_c\|^2}=\frac{\|\s_c\|}{\|\y_c\|}
        =\frac{\|\s_c\|^2}{\|\y_c\|\|\s_c\|}
      \leq \frac{\|\s_c\|^2}{\y_c^\T\s_c} \leq \frac{4}{\kappa_{\epsilon}},
  \end{equation*}
  which means that three candidates $\gamma_c^{\mathrm{BB1}}$, $\gamma_c^{\mathrm{BB2}}$,
  and $\gamma_c^{\mathrm{Dai}}$ satisfy the inequality \eqref{gamma-bd}.
\end{proof}

Third, based on Lemmas \ref{Lem:bfgs-ub}, \ref{Lem:bfgs-lb}, and \ref{Lem:gamma-bd},
we obtain two lemmas as follows.

\begin{Lemma}\label{Lem:Hk-ubd}
  Suppose that the approximation of a Hessian's inverse $H_c$
  is generated by L-BFGS \eqref{Hk-start}--\eqref{Hk-final}.
  Then, there exists a positive constant $C_U\geq 1$ such that
  \begin{equation*}
    \|H_c\| \leq C_U.
  \end{equation*}
\end{Lemma}
\begin{proof}
  From Lemma \ref{Lem:gamma-bd} and \eqref{Hk-start},
  we have $\|H_c^{(0)}\|\leq \frac{4}{\kappa_{\epsilon}}$.
  Then, by \eqref{Hk-final}, \eqref{Hk-rec} and Lemma \ref{Lem:bfgs-ub},
  we get
  \begin{eqnarray*}
    \|H_c\|&=&\|H_c^{(L)}\| \\
      &\leq& \|H_c^{(L-1)}\|\left(1+\frac{4M}{\kappa_{\epsilon}}\right)^2+\frac{4}{\kappa_{\epsilon}} \\
      &\leq& \cdots \\
      &\leq& \|H_c^{(0)}\|\left(1+\frac{4M}{\kappa_{\epsilon}}\right)^{2L} +
        \frac{4}{\kappa_{\epsilon}}\sum_{\ell=0}^{L-1}\left(1+\frac{4M}{\kappa_{\epsilon}}\right)^{2\ell} \\
      &\leq& \frac{4}{\kappa_{\epsilon}}\sum_{\ell=0}^L\left(1+\frac{4M}{\kappa_{\epsilon}}\right)^{2\ell}
        \equiv C_U.
  \end{eqnarray*}
  The proof is complete.
\end{proof}

\begin{Lemma}\label{Lem:Hk-lbd}
  Suppose that the approximation of a Hessian's inverse $H_c$
  is generated by L-BFGS \eqref{Hk-start}--\eqref{Hk-final}.
  Then, there exists a constant $0< C_L <1$ such that
  \begin{equation*}
    \mu_{\min}(H_c) \geq C_L.
  \end{equation*}
\end{Lemma}
\begin{proof}
  From Lemma \ref{Lem:gamma-bd} and \eqref{Hk-start},
  we have $\mu_{\min}(H_c^{(0)})\geq \frac{\kappa_{\epsilon}}{4M^2}$.
  Moreover, Lemma \ref{Lem:Hk-ubd} means that
  $\|H_c^{(\ell)}\|\leq C_U$ for all $\ell=1,\ldots,L$.
  Hence, Lemma \ref{Lem:bfgs-lb} implies
  \begin{equation*}
    \mu_{\min}(H_c^{(\ell+1)}) \geq
      \frac{\kappa_{\epsilon}}{\kappa_{\epsilon}+4M^2C_U}\mu_{\min}(H_c^{(\ell)}).
  \end{equation*}
  Then, from \eqref{Hk-final} and \eqref{Hk-rec}, we obtain
  \begin{eqnarray*}
    \mu_{\min}(H_c)&=&\mu_{\min}(H_c^{(L)}) \\
      &\geq&\frac{\kappa_{\epsilon}}{\kappa_{\epsilon}+4M^2C_U}\mu_{\min}(H_c^{(L-1)}) \\
      &\geq& \cdots \\
      &\geq&\left(\frac{\kappa_{\epsilon}}{\kappa_{\epsilon}+4M^2C_U}\right)^L\mu_{\min}(H_c^{(0)}) \\
      &\geq&\frac{\kappa_{\epsilon}}{4M^2}
        \left(\frac{\kappa_{\epsilon}}{\kappa_{\epsilon}+4M^2C_U}\right)^L
        \equiv C_L.
  \end{eqnarray*}
  We complete the proof.
\end{proof}

Finally, the proof of Theorem \ref{Th: grad-dir} is straightforward
from Lemmas \ref{Lem:Hk-ubd} and \ref{Lem:Hk-lbd}.



\end{document}